\documentclass[11pt,A4paper]{article}

\usepackage[left=32mm,right=32mm,top=30mm,bottom=32mm]{geometry}
\usepackage{labelfig}
\usepackage{epsfig}
\usepackage{epstopdf}
\usepackage{psfrag}
\usepackage{xfrac}

\usepackage[all]{xy}

\usepackage{color}

\usepackage{amsthm,amsmath,amssymb}
\usepackage{booktabs}
\usepackage{txfonts}
\usepackage{mathpazo}
\usepackage{microtype}
\usepackage{overpic}
\usepackage{bm}
\usepackage{sectsty}
\usepackage[
	pdftitle={PDFTitle},
	pdfauthor={Hugo Parlier},
	ocgcolorlinks,
	linkcolor=linkred,
	citecolor=linkred,
	urlcolor=linkblue]
{hyperref}

\definecolor{linkred}{rgb}{0.48,0.1,0.05}
\definecolor{linkblue}{RGB}{16, 78, 139}

\usepackage[hang,flushmargin]{footmisc}
\usepackage{enumitem}
\usepackage{titlesec}
	\titlespacing{\section}{0pt}{12pt}{0pt}
	\titlespacing{\subsection}{0pt}{6pt}{0pt}
	
\titlelabel{\thetitle.\quad}

\makeatletter 

\long\def\@footnotetext#1{%
\H@@footnotetext{%
\ifHy@nesting 
\hyper@@anchor{\@currentHref}{#1}%
\else 
\Hy@raisedlink{\hyper@@anchor{\@currentHref}{\relax}}#1%
\fi 
}}

\def\@footnotemark{%
\leavevmode 
\ifhmode\edef\@x@sf{\the\spacefactor}\nobreak\fi 
\H@refstepcounter{Hfootnote}%
\hyper@makecurrent{Hfootnote}%
\hyper@linkstart{link}{\@currentHref}%
\@makefnmark 
\hyper@linkend 
\ifhmode\spacefactor\@x@sf\fi 
\relax 
}%

\ifFN@multiplefootnote%
\renewcommand*\@footnotemark{%
\leavevmode 
\ifhmode 
\edef\@x@sf{\the\spacefactor}%
\FN@mf@check 
\nobreak 
\fi 
\H@refstepcounter{Hfootnote}%
\hyper@makecurrent{Hfootnote}%
\hyper@linkstart{link}{\@currentHref}%
\@makefnmark 
\hyper@linkend 
\ifFN@pp@towrite 
\FN@pp@writetemp 
\FN@pp@towritefalse 
\fi 
\FN@mf@prepare 
\ifhmode\spacefactor\@x@sf\fi 
\relax%
}%
\fi 

\makeatother 

\theoremstyle{plain}
\newtheorem{theorem}{Theorem}[section]
\newtheorem{proposition}[theorem]{Proposition}
\newtheorem{lemma}[theorem]{Lemma}
\newtheorem{corollary}[theorem]{Corollary}

\theoremstyle{definition}
\newtheorem{definition}[theorem]{Definition}

\newtheorem{remark}[theorem]{Remark}

\newcommand{\F}{{\mathcal F}}

\newcommand{\card}{{\rm card}}

\usepackage{mathtools}

\DeclarePairedDelimiter\floor{\lfloor}{\rfloor}

\newcommand{\diam}{\ensuremath{\mathrm{diam}}}
\newcommand{\cone}{\ensuremath{\mathrm{cone}}}

\newcommand{\MOD}{\ensuremath{\mathrm{Mod}}}
\newcommand{\Stab}{\ensuremath{\mathrm{Stab}}}

\newcommand{\A}{\mathcal A}

\newcommand{\MF}{\mathcal M \mathcal F}

\sectionfont{\large \sc}
\subsectionfont{\normalsize}

\long\def\symbolfootnote[#1]#2{\begingroup%
\def\thefootnote{\fnsymbol{footnote}}\footnote[#1]{#2}\endgroup}

\def\blfootnote{\xdef\@thefnmark{}\@footnotetext}

\begin{document}

{\Large \bfseries \sc The geometry of flip graphs and mapping class groups}

{\bfseries Valentina Disarlo\symbolfootnote[2]{\normalsize Research partially funded by an International Scholarship from the University of Fribourg.\vspace{0.2cm}}, 
Hugo Parlier\symbolfootnote[1]{\normalsize Research supported by Swiss National Science Foundation grants numbers PP00P2\textunderscore 15302 and PP00P2\textunderscore 128557.
}}

{\em Abstract.} The space of topological decompositions into triangulations of a surface has a natural graph structure where two triangulations share an edge if they are related by a so-called flip. This space is a sort of combinatorial Teichm\"uller space and is quasi-isometric to the underlying mapping class group. We study this space in two main directions. We first show that strata corresponding to triangulations containing a same multiarc are strongly convex within the whole space and use this result to deduce properties about the mapping class group. We then focus on the quotient of this space by the mapping class group to obtain a type of combinatorial moduli space. In particular, we are able to identity how the diameters of the resulting spaces grow in terms of the complexity of the underlying surfaces. 

\vspace{1cm}

\section{Introduction}\label{sec:intro}

The many relationships between curves, arcs and homeomorphisms of surfaces have provided numerous, rich and fruitful insights into the study of Teichm\"uller spaces and mapping class groups. In particular, combinatorial structures such as curve, arc and pants complexes have been shown to be closely related to metric structures on Teichm\"uller spaces and in particular all share the mapping class group as an automorphism group. 

Flip graphs are an example of one of these natural combinatorial structures. For a given topological surface with a prescribed set of points, the vertex set of the associated flip graph is the set of maximal multiarcs (which have begin and terminate in the prescribed points). Just like the other combinatorial objects, the multiarcs are considered up to isotopy (which preserve the prescribed set of points). As they are maximal, they decompose the surface into triangles and thus we refer to them as triangulations (although they may not be triangulations in the usual sense). Two triangulations share an edge in the flip graph if they are related by a flip - so if they differ by at most one arc. Provided the surface is complicated enough, flip graphs are infinite objects but are always locally finite connected graphs. 

Flip graphs can be thought of (and this is the point of view we take in this article) as metric objects by associating length 1 to every edge. As metric spaces they describe how different triangulations are from one another and are a sort of combinatorial analogue of Teichm\"uller space. With a few exceptions, the mapping class group is again the full automorphism group \cite{Kork-Pap} and as such the finite quotient, which we call a {\it modular flip graph}, becomes a combinatorial analogue of a moduli space. In contrast to some of the other spaces mentioned before, the action of the mapping class group is proper and, via the \v{S}varc-Milnor lemma, the flip graph is a quasi-isometric model of the mapping class group which makes it an ideal tool for studying its geometry. Mosher \cite{Mosher2} implicitly uses the flip graph to study the mapping class group from the combinatorial point of view. This point of view has recently been exploited by Rafi and Tao \cite{RT2}. 

Flip graphs of surfaces also appear in a number of other contexts. As hinted at above, flip graphs naturally appear in Teichm\"uller theory. They appear for instance in Penner's decorated Teichm\"uller space \cite{Pen3} and in the work of Fomin, Shapiro and D. Thurston (\cite{FST1} and \cite{FST2}) in their study of cluster algebras related to bordered surfaces. Flip graphs and some slight variations have been studied in combinatorics and computational geometry by a variety of authors, for instance Negami \cite{Negami1}, Bose \cite{Bose1} and De Loera-Rambau-Santos \cite{triang-book}. 

One of the simplest and most studied flip graphs is the flip graph of a polygon, the so-called associahedron \cite{Stasheff, Tamari}. It is a finite graph with a number of remarkable properties including being the graph of a polytope. The celebrated result of Sleator, Tarjan and W. Thurston \cite{STT2} about the diameter of the associahedron, and proved using 3-dimensional hyperbolic polyhedra, was recently extended by Pournin \cite{Pournin} who also provided a purely combinatorial proof. The diameter of this graph is exactly $2n-10$ for all $n>12$. Sleator, Tarjan, and W. Thurston \cite{STT1} also studied triangulations of spheres up to homeomorphism, which essentially amounts to studying the diameter of a modular flip graph. In this case, they show that the diameter grows like $n \log n$ where $n$ is the number of labelled points on the sphere. 

In this article, we study both the geometry of flip graphs and of modular flip graphs. One of the main motivations we have in mind is the study of the mapping class group. 

We begin by studying the geometry of flip graphs. Our first main result comes from a very natural question about two triangulations that have an arc $a$ in common. Given any two such triangulations, there is at least one minimal path between them: do all the triangulations of any minimal path contain the arc $a$? The answer is yes. 
\begin{theorem}\label{thm:convex}
For every multiarc $A$, the stratum $\F_A$ is strongly convex. 
\end{theorem}
In the above result, for any given flip graph, we've denoted $\F_A$ the set of triangulations which contained a prescribed multiarc $A$. We note that this result for flip graphs of polygons was previously known and an essential tool in \cite{STT1} and in \cite{Pournin}. 

We observe that the same question can be asked for the pants graph (where multicurves play the part of multiarcs). For the pants graph, this is known to be true for certain types of multicurves but is in general completely open \cite{APS1, APS2, ALPS, TaylorZupan}. 

We give two applications of this result. It is a recent result of the second author together with Aramayona and Koberda that, under certain conditions, simplicial embeddings between flip graphs only arise {\it naturally} \cite{AKP}. By naturally, we mean that the injective simplical map comes from an embedding between the two surfaces. The conditions are on surface in the domain flip graph that is required to be non-exceptional (or ``sufficiently complicated", see Section \ref{ss:applications} for a precise definition). Now together with the above theorem, this implies the following. 
\begin{corollary}
Suppose $\Sigma$ is non-exceptional, and let $\F(\Sigma) \to \F(\Sigma')$ be an injective simplicial map. Then $\F(\Sigma)$ is strongly convex inside of $ \F(\Sigma')$.
\end{corollary}

As geometric properties of flip graphs translate into a quasi properties for mapping class groups, we also obtain the following result for mapping class groups. This result also follows by results of Masur-Minsky \cite{MM2} and Hamenst\"adt \cite{Ham2}.

\begin{corollary}\label{cor:convex1}
For every vertex $T \in \F_A$, there is a commutative diagram: 
$$\xymatrix{
\F_A \ar@{^{(}->}[r] &\F(\Sigma) \\
\mathrm{Stab}(A) \ar[u]^{{\omega_T}_|} \ar@{^{(}->}[r] & \MOD(\Sigma) \ar[u]_{\omega_T} }
$$
where the inclusion $\F_A \hookrightarrow \F(\Sigma)$ is an isometry and the orbit map $\omega_T: \MOD(\Sigma) \to \F(\Sigma)$ restricts to a quasi-isometry ${\omega_T}_|: \Stab(A) \to \F_A $. 
Moreover, the inclusion $\mathrm{Stab}(A) \hookrightarrow \MOD(\Sigma)$ is a quasi-isometric embedding. 
\end{corollary}

After these results about the geometry of flip graphs and mapping class groups, we shift our focus to the quotient of the former by the latter, namely the geometry of modular flip graphs $\MF(\Sigma)$. In particular, we study their diameter and how it grows in function of the topology of the base surface. Our main results are upper and lower bounds that have the same growth rates in terms of the number of punctures and genus. We summarize them in the following theorem.

\begin{theorem}\label{thm:diameters}
There exist constants $L>0$ and $U>0$ such that if $\Sigma$ be a surface of genus $g$ with $n$ labelled punctures then
$$
L \left( g \log(g+1) + n \log(n+1)\right)\leq \diam\left(\MF(\Sigma)\right) \leq U \left( g \log(g+1) + n \log(n+1)\right).
$$
\end{theorem}
The above result is a combination of results (namely Theorems \ref{thm:uppergenus}, \ref{thm:upperpuncture} and Corollary \ref{cor:card}) from which the constants $L,U$ can be made explicit. When the punctures are not labeled, we obtain similar results and this time the growth rate is linear in $n$ (Theorem \ref{thm:unmarked} and Corollary \ref{cor:count}). 

We note that this result is a generalization of the result of Sleator, Tarjan and W. Thurston mentioned above about the diameters of modular flip graphs of punctured spheres and in fact our lower bounds are obtained using a counting argument and one of their results. Our result also provides a lower bound on the diameters of some slight refinements of the flip graph used in computational geometry and combinatorics. Indeed, it follows that the distance between any two simple triangulations (i.e. not containing multiple edges or loops) of a surface with labelled punctures grows at least like
$$n\log(n) + g\log(g).$$
This can be compared with results of Negami \cite{Negami2, Negami1} and Cortes et al. \cite{Hurtado}. 

We also note that the growth rates are reminiscent of the growth rates of a type of combinatorial moduli space related to cubic graphs. More precisely, one can endow the set of isomorphism types of cubic graphs with $m$ vertices with a metric where one counts the minimal number of {\it Whitehead moves} (or $\tilde{S}$-transformations in the language of \cite{tsukui}). We refer the reader to \cite{Cavendish} or \cite{RT} for the definitions of these terms. With this metric, the diameter of this space is also of rough growth $m \log m$ (Cavendish \cite{Cavendish}, Cavendish-Parlier \cite{CP} and Rafi-Tao \cite{RT}). 

Dual to a triangulation is a cubic graph and flips correspond to specific types of Whitehead moves. One might think in first instance that the two results are in fact the same, but one does not seem to imply the other. On the one hand, flipping only allows for certain moves so the result on flip graphs certainly seems stronger. However, given two cubic graphs with the same number of vertices, there is no guarantee that they are both the dual graph triangulations that lie in the same flip graph. 

This article is organized as follows. 

In the preliminary section, we provide detailed descriptions of the objects we study and some known results. We also prove a number of preliminary results including for instance a new algorithm to reach a stratum with distance bounded by the intersection number. In particular this provides a new proof that intersection number bounds the flip distance between two triangulations. We also provide a lower bound on distance in terms of intersection number. We conclude the section with two results that are somewhat parallel to the rest of the paper about the mapping class group and flip graphs. As far as we know, although both are known, our proofs are new. We provide these results to illustrate the point that flip graphs can be used to effectively study the mapping class group. 

In the third section, we prove Theorem \ref{thm:convex} stated above. We then provide two applications of this result. The first is about projections to strata and the second is to the large scale geometry of the mapping class group as discussed above. 

The final section is about the diameters of modular flip graphs. We begin with upper bounds - first in terms of genus and then in terms of the number of punctures - and we end with the lower bounds.

{\bf Acknowledgements.}

Part of this work was carried out while the first author was visiting the second author at the University of Fribourg. She is grateful to the department and the staff for the warm hospitality. She also acknowledges the support of Indiana University Provost's Travel Award for Women in Science.

The authors acknowledge support from U.S. National Science Foundation grants DMS 1107452, 1107263, 1107367 "RNMS: GEometric structures And Representation varieties" (the GEAR Network).

We would like to thank Javier Aramayona, Chris Connell, Chris Judge, Chris Leininger, Athanase Papadopoulos, Bob Penner, Lionel Pournin and Dylan Thurston for their encouragement and enlightening conversations. 

\section{Preliminaries}

In this section we describe in some detail the objects we are interested in and introduce tools we use in the sequel. Most of the results we state are already known, although some of the proofs we provide are new (or at least we did not find them in the literature). In particular, at the end of this section we give two quick examples of results one can prove using flip graphs. Neither are essential in the sequel and are just provided for illustrative purposes. 

\subsection{Definitions and setup}

We begin with the basic setup which starts with a topological orientable connected surface $\Sigma$ and finite set of marked points on it. Unless specifically stated, $\Sigma$ will be assumed to be triangulable. It is of finite type, has boundary which can consist of marked points, and boundary curves, and each boundary curve must have at least one marked point on it. We make the distinction between {\it labelled} and {\it unlabelled} marked points when we look at how homeomorphisms are allowed to act on $\Sigma$ - this will made explicit in what follows. 

Sometimes marked points that do not lie on a boundary curve will be referred to as {\it punctures}.

To such a $\Sigma$ one can associate its {\it arc complex} $\A(\Sigma)$, a simplicial complex where vertices are isotopy classes of simple {\it arcs} based at the marked points of $\Sigma$. Simplices are spanned by {\it multiarcs} (unions of isotopy classes of arcs disjoint in their interior). We won't explicitly use this complex so we won't describe it in full detail, but we will be interested in the graph which is the $1$-skeleton of the cellular complex dual to $A(\Sigma)$: the flip graph of $\Sigma$. 

The flip graph $\F(\Sigma)$ can be described differently as follows. Vertices of this graph are maximal multiarcs so they decompose $\Sigma$ into triangles. We refer to these multiarcs as {\it triangulations} (although they are not always proper triangulations in the usual sense - we apologize any confusion which incurs from this by quite common terminology).

Two vertices of $\F(\Sigma)$ share an edge if they differ by a so-called {\it flip}. If $a$ is an arc of a triangulation $T$ which belongs to two triangles which form a quadrilateral, a {\it flip} is the operation which consists of replacing $a$ by the other diagonal arc $a'$ of the quadrilateral. 

\begin{figure}[h]
\begin{center}
\includegraphics[width=6cm]{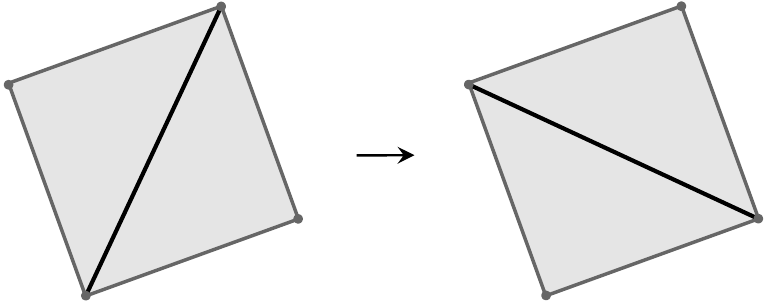}
\caption{A local picture of a flip}
\label{fig:flip}
\end{center}
\end{figure}

Note that certain arcs are not {\it flippable} - this occurs exactly when an arc is contained in a punctured disc surrounded by another arc.

\begin{figure}[h]
\begin{center}
\includegraphics[width=3.0cm]{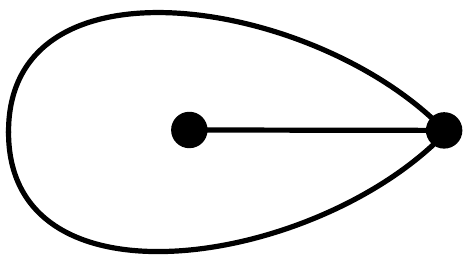}
\caption{The central arc is unflippable}
\label{fig:unflippable}
\end{center}
\end{figure}

We denote by $\kappa = \kappa(\Sigma)$ the number of arcs in (any) triangulation of $\Sigma$ and by $\tilde{\kappa}=\tilde{\kappa}(\Sigma)$ the number of triangles in the complement of any triangulation of $\Sigma$. Via an Euler characteristic argument, one obtains $\tilde{\kappa}(\Sigma) = 4g + 4b +2s +p -6$ and $\kappa(\Sigma)= 6g + 3b + 3s + p - 6$, where $g$ is the genus of $\Sigma$, $s$ is the number of punctures, $b$ is the number of boundary components and $p$ is the number of points on the boundary curves of $\Sigma$.

Some flip graphs are finite - such as the flip graph of a polygon - but provided the underlying surface has enough topology, $\F(\Sigma)$ is infinite. A simple example of an infinite flip graph is given by the flip graph of a cylinder with a single marked point on each boundary curve. By the above formula, vertices of $\F(\Sigma)$ are of degree $2$, and as it is both infinite and connected, it is isomorphic to the infinite line graph. 

Associated to a multiarc $A$ is a {\it stratum} $\F_A(\Sigma)$ which is the flip graph of triangulations of $\Sigma$ which contain the multiarc $A$. We say that a stratum is {\it strongly convex} if any geodesic between two of its points is entirely contained in the stratum (sometimes this property is referred to as being {\it totally geodesic}). 

Naturally, if $A$ is not separating, $\F_A(\Sigma)$ is isomorphic to the flip graph of $\Sigma \setminus A$ (the surface {\it cut along} the multiarc $A$). There is something to be said here - we think of the result of the operation of cutting not as being the deletion of the arcs but by doubling the arcs and then separating them. For instance, if you cut a once punctured torus along an arc, the result is a cylinder with a marked point on each boundary curve. If $A$ is separating, $\F_A(\Sigma)$ is isomorphic to the product of the product of the flip graphs of the connected components of $\Sigma \setminus A$. In the rest of the paper we will denote by $|A|$ the number of arcs in $A$.

As arcs are thought of as isotopy classes of arcs, the intersection $i(a,b)$ between arcs $a$ and $b$ is defined to be the minimum number of intersection points between their representatives. Generally we assume arcs and multiarcs to be realized in minimal position. If $A$ and $B$ are multiarcs, their \emph{intersection number} is defined as 
$$i(A,B) = \sum_{b \in B} \sum_{a\in A} i(a,b).$$

In terms of intersection, two triangulations $T,T'$ are related by a flip if they satisfy
$$
i(T,T')=1.
$$

The flip graph is known to enjoy a number of properties. 

First of all, for any topological type of $\Sigma$, $\F(\Sigma)$ is a connected graph. There are several different proofs of this fact (see for instance Hatcher \cite{Hat}). We will consider the edges of the flip graph of length 1 and we will endow the flip graph with its shortest path distance. The distance between two triangulations is then equal to the minimum number of flips required to pass from one to the other. In particular there is the following quantitative version relating distance and intersection number which can be deduced from an algorithm described by Mosher in \cite{Mosher1} and Penner in \cite{Pen5}. 
\begin{lemma}\label{lem:mosher}
For any triangulation $S,T \in \F(\Sigma)$ we have $d(S,T) \leq i(T,S)$. 
\end{lemma}
We will give an alternative proof of this lemma in Section \ref{ss:upper}.

The homeomorphisms of $\Sigma$ considered here always fix pointwise the labelled points of $\Sigma$ and setwise the unlabelled. Permutations of the unlabelled points are allowed. The \emph{mapping class group} $\MOD(\Sigma)$ of $\Sigma$ is the group of orientation preserving homeomorphisms of $\Sigma$ up to isotopy. Isotopies here always fix pointwise the set of the marked points of $\Sigma$. The group $\MOD(\Sigma)$ act simplicially by automorphisms on $\F(\Sigma)$. It is a result of Korkmaz and Papadopoulos \cite{Kork-Pap} that except for some low complexity cases, the automorphism group of $\F(\Sigma)$ is exactly the {\it extended mapping class group} of $\Sigma$ - i.e. the group of homeomorphisms up to isotopy (orientation reversing homeomorphisms are also allowed). Related to this result, is a result about subgraphs of flip graphs that are graph isomorphic to other flip graphs. Except for some complexity cases again, such subgraphs only arise in the natural way - as strata associated to a certain multiarc \cite{AKP}.

We will be interested in the geometry of the flip graph as a metric space. We recall a few notions of metric geometry that we will use later in the paper. 

Let $(X, d_X)$ and $(Y, d_Y)$ be two metric spaces. A map
$$f:(X,d_X) \to (Y,d_Y)$$
is a $(\lambda, \epsilon)$-\emph{quasi-isometric embedding} if for all $x, x' \in X$ we have
$$\frac{1}{\lambda} d_X(x, x') - \epsilon \leq d_Y(f(x), f(x')) \leq \lambda d_X(x, x') + \epsilon.$$
We say that a quasi-isometric embedding $f$ is a \emph{quasi-isometry} there exists $R \geq 0$ such that the image of $f$ is $R$-dense in $Y$. Equivalently, $f$ is a quasi-isometry if there exists a quasi-isometric embedding $g: Y \to X$ and a constant $K\geq 0$ such 
that for all $x \in X$ and for all $y \in Y$ we have $d_X( g\circ f(x), x) \leq K$ and $d_Y( f\circ g(y), y) \leq K$. We say that $g$ is a \emph{quasi-inverse} of $f$. 

The following lemma is a classic result in geometric group theory. 
\begin{lemma}[\v{S}varc-Milnor]
Let $G$ be a group acting on a metric space $(X, d_X)$ properly and cocompactly by isometries. Then $G$ is finitely generated and for every $x\in X$ the orbit map 
\begin{align*}
\omega_X : G & \to X \\
g &\mapsto g(x)
\end{align*}
is a quasi-isometry. 
\end{lemma}

The mapping class group acts on the flip graph by isometries. The \v{S}varc-Milnor lemma applies directly to the flip graph and the mapping class group, and of course is only interesting when $\F(\Sigma)$ is of infinite diameter. 
\begin{lemma}\label{mcg-flipgraph}
For every triangulation $T \in \MOD(\Sigma)$ the orbit map 
\begin{align*}
\omega_T : \MOD(\Sigma) & \to \F(\Sigma) \\
g &\mapsto g(T)
\end{align*}
is a quasi-isometry. 
\end{lemma}

\begin{proof}
We will use the \v{S}varc-Milnor Lemma. The action of $\MOD(\Sigma)$ on $\F(\Sigma)$ is cocompact since there is only a finite number of ways to glue $\tilde{\kappa}$ triangles to get a surface homeomorphic to $\Sigma$. 
For a triangulation $T \in \F(\Sigma)$, we denote by $\Stab(T)$ its stabilizer in $\MOD(\Sigma)$ (the group of mapping classes that fix $T$ setwise). We will prove that for every $T$ the stabilizer $\Stab(T)$ is finite and this suffices to prove that the action of $\MOD(\Sigma)$ on $\F(\Sigma)$ is proper. Indeed, every mapping class in $\Stab(T)$ induces a permutation of the arcs in $T$, and there is a short sequence of groups
$$ 1 \to \Stab(T) \to \mathcal S_\kappa $$
where $\mathcal S_\kappa$ is the symmetric group on $\kappa$ elements. 
The sequence is exact since a mapping class that fixes every arc of a triangulation is the identity by the Alexander lemma (see for instance \cite{FM}).

The last assertion follows directly from the \v{S}varc-Milnor lemma.
\end{proof}

We define the \emph{modular flip graph} $\MF(\Sigma)$ as the quotient of $\F(\Sigma)$ under the action of $\MOD(\Sigma)$. We remark that points in $\MF(\Sigma)$ are triangulations of $\Sigma$ up to homeomorphisms. By the above lemma $\MF(\Sigma)$ is a connected finite graph that inherits a well-defined distance from $\F(\Sigma)$. We note that an orbit map in Lemma \ref{mcg-flipgraph} is $ \diam\left(\MF(\Sigma)\right)$-dense in $\F(\Sigma)$ by the \v{S}varc-Milnor Lemma. We will later investigate the diameter of $\MF(\Sigma)$.

The following result will be a helpful tool in our computation. This result was first proved by Sleator-Tarjan-Thuston \cite{STT2} provided that $n$ is large enough. Recently Pournin \cite{Pournin} provided a combinatorial proof and proved the lower bounds for all $n>12$. 
\begin{theorem}\label{th:STT}
If $\Sigma$ is a disk with $n > 12$ labelled points on the boundary then $\F(\Sigma)$ has diameter $2n - 10$. 
\end{theorem}

We finally note that an orbit map in Lemma \ref{mcg-flipgraph} is $\diam( \MF(\Sigma))$-dense in $\F(\Sigma)$ by the \v{S}varc-Milnor Lemma. 

\subsection{Intersection number and distances}

\subsubsection{An upper bound}\label{ss:upper}
In this section we describe an algorithm to get from a triangulation $T$ to a stratum associated to a multiarc $A$. To do this, we prove that there exists an arc in $T$ that intersects $A$ maximally and such that its flip reduces the number of intersections with $A$. This provides an alternative proof of Lemma \ref{lem:mosher} above. 

\begin{definition}
Let $\Delta$ be a triangle in $\Sigma \setminus T$ and $A$ a multiarc. We say that $\Delta$ is \emph{terminal} for $A$ if there exists $a \in A$ such that $\Delta$ is the first or the last triangle crossed by $a$ (see Figure \ref{fig:terminal} for an example). 

\begin{figure}[h]
\leavevmode \SetLabels
\L(.51*.56) $a$\\
\endSetLabels
\begin{center}
\AffixLabels{\centerline{\epsfig{file =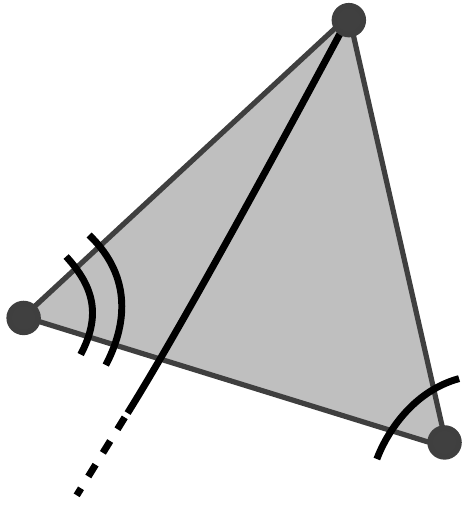,width=3.5cm,angle=0} }}
\vspace{-30pt}
\end{center}
\caption{A terminal triangle for $a$} \label{fig:terminal}
\end{figure}
\end{definition}

\begin{definition}
Let $T$ be a triangulation and $A$ a multiarc. Let $t \in T$ be a flippable edge and $T'$ the triangulation obtained by the flip of $t$. 
We say that flipping $t$ is \emph{convenient} if 
$i(T', A) < i(T, A)$ and flipping $t$ is \emph{neutral} if $i(T', A) = i(T, A)$. 
\end{definition}
Denote by $t'$ is the arc obtained flipping $t$. Note that flipping $t$ is convenient if and only if $i(t',A) < i(t,A)$. Similarly, flipping $t$ is neutral if and only if $i(t',A) = i(t,A) $. Also note that an arc may be neither convenient nor neutral. 

\begin{lemma}\label{flip-convenient-0}
Assume $i(h,A) >0$. If $h\in T$ is an arc such that $i(h, A) = \max_{t \in T} i(t, A)$, then $h$ is flippable in $T$. 
\end{lemma}

\begin{proof}
If $h$ is not flippable then there exists an arc $h^\star \in T$ that surrounds $h$ and bounds a once-punctured disk. Hence $i(h^\star, A) \geq 2 i(h, A)$ in contradiction with our condition on $h$. 
\end{proof}

\begin{lemma}\label{flip-convenient-1}
Assume $i(T,A) >0$. Let $h \in T$ be an arc such that $i(h, A) = \max_{t \in T} i(t, A)$. Let $Q$ be the quadrilateral containing $h$ as a diagonal. 
If $Q$ contains a terminal triangle for $A$ then flipping $h$ is convenient.
\end{lemma}

\begin{proof}

Let $a\in A$ be an arc that terminates on $Q$. We begin by observing that if $h$ is not the first arc of $T$ that $a$ crosses from its terminal point, $h$ is not maximal. Indeed, if $h'\in T$ is the first arc crossed, then any arc that crosses $h$ is forced to cross $h'$ and it has 

$$i(h',A) \geq i(h,A)+1,$$
in contradiction with the maximality of $h$ (see Figure \ref{fig:Hmax1}).

\begin{figure}[h]
\leavevmode \SetLabels
\L(.505*.45) $h$\\
\L(.45*.08) $h'$\\
\endSetLabels
\begin{center}
\AffixLabels{\centerline{\epsfig{file =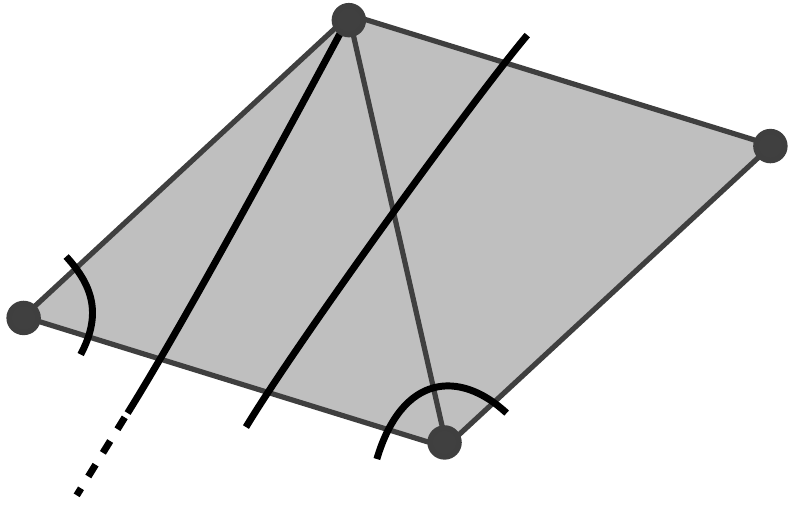,width=5.0cm,angle=0} }}
\vspace{-30pt}
\end{center}
\caption{If $h$ is the first arc of $T$ crossed by $a$, it cannot be maximal} \label{fig:Hmax1}
\end{figure}
We now prove that if $h$ is the first (or the last) edge crossed by an arc in $A$ (as in Figure \ref{fig:figura1}), then flipping $h$ is convenient. Let $h'$ be the arc obtained from flipping $h$. Let us now compute $i(h', A) - i(h, A) $. 

Up to possible symmetries, we may assume that $A$ contains at least an arc that crosses $z$ and $h$ and terminates in the top vertex of $Q$. Up to isotopy, only three configurations of $A$ and $Q$ are possible, and these are described in Figure \ref{fig:figura1}. We will use the following notation: 
\begin{itemize}
\item $\epsilon$ is the number of arcs in $A\cap Q$ that terminate in the top vertex of $Q$, crossing $z$ and $h$. Under our assumption, $\epsilon \geq 1$. 
\item $\epsilon'$ is the number of those that terminate in the top vertex of $Q$, crossing $w$ and $h$; 
\item $\epsilon''$ is the number of those that terminate in the bottom vertex of $Q$, crossing $h$ and $y$;
\item $\alpha $ is the number of arcs in $A\cap Q$ that wrap around the left endpoint of $h$ crossing $z,h$ and $x$; 
\item $\beta$ is the number of arcs in $A \cap Q$ that wrap around the right endpoint of $h$ crossing $w,h$ and $y$; 
\item $\gamma$ is the number of arcs in $A\cap Q$ that wrap around the bottom vertex of $Q$, crossing $z$, $h$ and $w$;
\item $\eta$ is the number of arcs that cross $z$, $h$ and $y$. 
\end{itemize}

\begin{figure}[h]
\leavevmode \SetLabels
\L(.07*.76) $x$\\
\L(.2*.9) $y$\\
\L(.12*.18) $z$\\
\L(.21*.28) $w$\\
\L(.23*.45) $h$\\
\L(.004*.42) $\alpha$\\
\L(.313*.39) $\beta$\\
\L(.18*.05) $\gamma$\\
\L(.262*.81) $\eta$\\
\L(.04*.27) $\epsilon$\\
\L(.404*.76) $x$\\
\L(.534*.9) $y$\\
\L(.456*.18) $z$\\
\L(.544*.28) $w$\\
\L(.58*.45) $h$\\
\L(.338*.42) $\alpha$\\
\L(.647*.39) $\beta$\\
\L(.62*.7) $\epsilon''$\\
\L(.596*.81) $\eta$\\
\L(.374*.27) $\epsilon$\\
\L(.735*.76) $x$\\
\L(.865*.9) $y$\\
\L(.785*.18) $z$\\
\L(.855*.22) $w$\\
\L(.91*.45) $h$\\
\L(.669*.42) $\alpha$\\
\L(.978*.39) $\beta$\\
\L(.18*.05) $\gamma$\\
\L(.915*.22) $\epsilon'$\\
\L(.705*.27) $\epsilon$\\
\endSetLabels
\begin{center}
\AffixLabels{\centerline{\epsfig{file =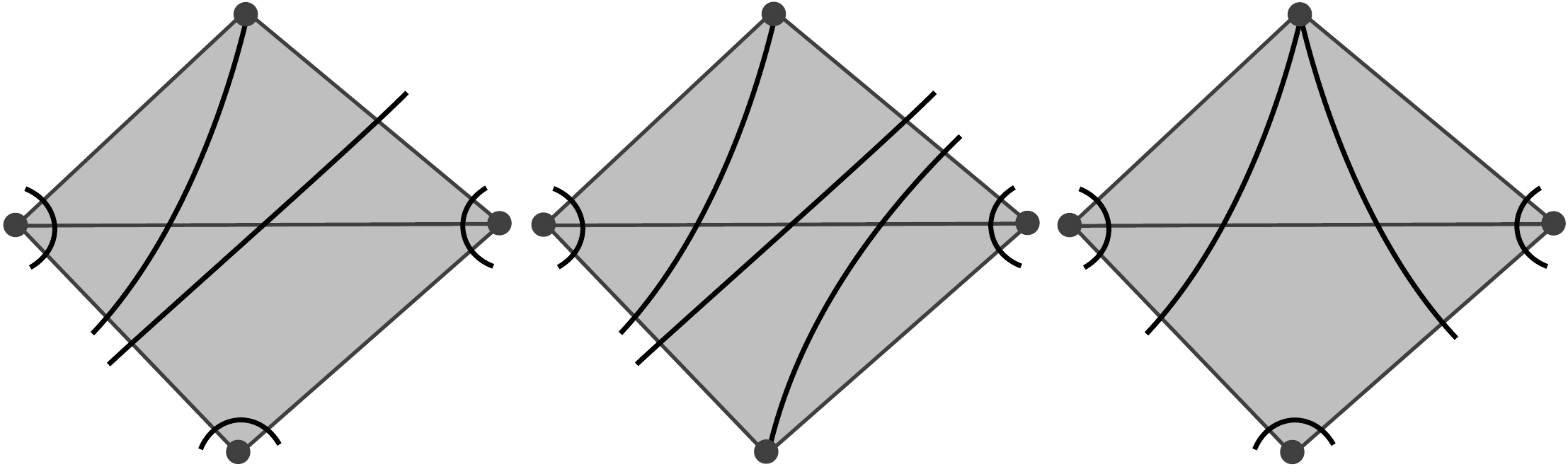,width=15.0cm,angle=0} }}
\vspace{-30pt}
\end{center}
\caption{The three possible configurations of arcs} \label{fig:figura1}
\end{figure}
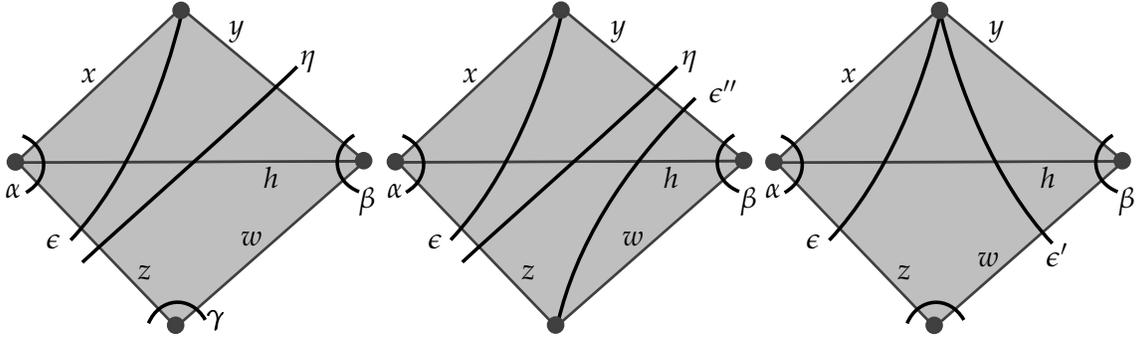

We remark that in Figure \ref{fig:figura1} - (1) we have $ \epsilon' = \epsilon'' = 0$, in Figure \ref{fig:figura1} - (2) we have $ \gamma =0$, in Figure \ref{fig:figura1} - (3) we have $\eta =0$. Moreover, in each configuration in Figure \ref{fig:figura1} the following holds: 
\begin{align*}
i(h,A) &= \alpha + \beta + \epsilon + \eta + \epsilon' + \epsilon'' \\
i(h', A) &= \gamma + \eta \\
i(z,A) &= \alpha + \eta + \gamma + \epsilon = \alpha + \epsilon + i(h',A)
\end{align*}

By definition of $h$, we have $i(z,A) \leq i(h,A)$. It follows: 
\begin{align*}
\alpha + \epsilon + i(h',A) &\leq i(h,A) \\ 
i(h',A) -i (h,A) &\leq - \alpha - \epsilon \leq -1. 
\end{align*} 
We conclude that flipping $h$ is convenient. 
\end{proof}

\begin{lemma}\label{flip-convenient-2}
Assume $i(T,A) >0$.
Let $h\in T$ be an arc such that $i(h, A) = \max_{t \in T} i(t, A)$. Let $Q$ be the quadrilateral containing $h$ as a diagonal, and 
assume that $Q$ does not contain a terminal triangle for $A$. 
Then flipping $h$ is neutral if and only if, according the notation of Figure \ref{fig:figura2}, $i(h, A) = i(y, A) = i(z,A)$. 

Moreover, if $i(h, A) \neq i(y, A)$ or $i(h, A) \neq i(z,A)$ then flipping $h$ is convenient.
\end{lemma}

\begin{proof}
We will compute $i(h',A) - i(h,A)$. Since $Q$ is not terminal for $A$, $A$ and $Q$ look like in Figure \ref{fig:figura2}, up to isotopy.

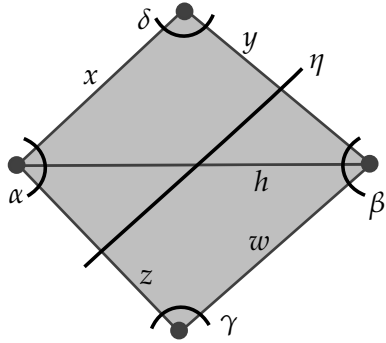
\begin{figure}[h]
\leavevmode \SetLabels
\L(.40*.76) $x$\\
\L(.54*.87) $y$\\
\L(.45*.18) $z$\\
\L(.545*.28) $w$\\
\L(.55*.45) $h$\\
\L(.335*.415) $\alpha$\\
\L(.65*.39) $\beta$\\
\L(.52*.05) $\gamma$\\
\L(.598*.805) $\eta$\\
\L(.448*.92) $\delta$\\
\endSetLabels
\begin{center}
\AffixLabels{\centerline{\epsfig{file =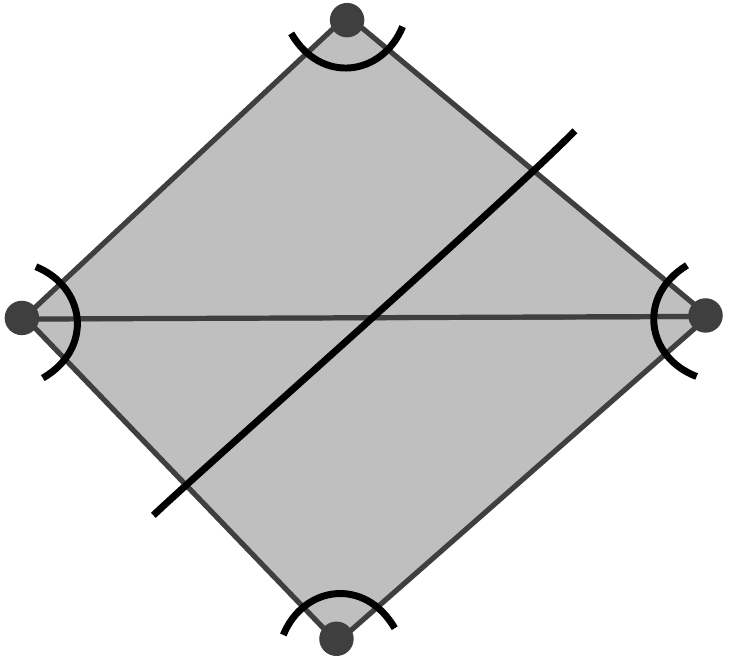,width=5.0cm,angle=0} }}
\vspace{-30pt}
\end{center}
\caption{Configuration of arcs} \label{fig:figura2}
\end{figure}

Denote by $\delta$ the number of arcs in $A\cap Q$ that wrap around the top corner of $Q$. In the notation of the proof of Lemma \ref{flip-convenient-1} we have: 
\begin{align*}
 i(h,A) &= \alpha + \eta + \beta \\ 
i(z,A) &= \alpha + \eta + \gamma \\
i(y,A) &= \eta + \beta + \delta \\
i(h',A) &= \gamma +\delta + \eta \\
i(h',A) - i(h,A) &= (\gamma - \beta) + (\delta - \alpha).
\end{align*}
By hypothesis on $h$, we have $i(z,A) \leq i(h,A)$, so $\gamma - \beta \leq 0$. Similarly, $i(y,A) \leq i(h,A)$, so $\delta - \alpha \leq 0$. It follows that 
$i(h', A) - i(h,A) = 0$ if and only if $\gamma = \beta$ and $\delta = \alpha$, that is, if and only if $i(z,A) = i(h,A)$ and $i(y,A) = i(h,A)$. We remark that in all the other cases $i(h',A) - i(h,A) <0$, and flipping $h$ is convenient.
\end{proof}

\begin{lemma}\label{flip-convenient-lemma}
If $i(T,A) > 0$, then there exists an arc $h \in T$ such that $i(h,A) = \max_{t \in T} i(t,A)$ and flipping $h$ is convenient. 
\end{lemma}

\begin{proof}
We will describe a procedure to find $h$. First pick an arc $m$ such that $i(m,A) = \max_{t \in T} i(t,A)$ . By Lemma \ref{flip-convenient-0} $m$ is flippable. If flipping $m$ is convenient, we set $h= m$ and we are done. Otherwise, by Lemma \ref{flip-convenient-1} the quadrilateral containing $m$ as a diagonal contains two more edges, $y$ and $z$, such that $i(y, A) = i(z, A) = i(m, A) = \max_{t \in T} i(t,A)$. Now set $m=y$ or $m=z$, and repeat this procedure. 
Lemma \ref{flip-convenient-1} ensures that the algorithm terminates. Indeed, if the quadrilateral containing $m$ as a diagonal also contains a terminal triangle for $A$ then flipping $m$ is convenient. 
\end{proof}

This now allows to describe a path from $T$ to $\F_A$. 
\begin{theorem}\label{our-algo}
Let $T$ be a triangulation and $A$ a multiarc. There is a sequence of triangulations $T = T_0 \rightarrow \ldots \rightarrow T_{i+1} \rightarrow \ldots$ iteratively constructed as follows:
\begin{enumerate} 
\item If $i(T_i, A) >0$, choose $h_i \in T_i$ as in Lemma \ref{flip-convenient-lemma}. Denote by $T_{i+1}$ the triangulation obtained flipping $h_i$.
\item If $i(T_i, A) =0$, terminate.
\end{enumerate}
Any such sequence constructed this way has at most $i(T,A)$ elements, and if $T_n$ is the terminal triangulation then $T_n \in \F_A$. 

Moreover, the above sequence satisfies $\max_{t \in T_{i+1}} i(t, A) \leq \max_{t \in T_{i}} i(t, A)$ for every $i = 0, \ldots, n$. 
\end{theorem}

\begin{proof}
By Lemma \ref{flip-convenient-lemma} flipping $h_i$ is convenient, so $$i(T_{i+1}, A) \leq i(T_i, A) - 1.$$ After at most $i(T,A)$ steps, we have $i(T_n, A) = 0$, that is, $T_n \in \F_A$. 
\end{proof}

From this the following corollary is immediate.
\begin{corollary}\label{dist-up}
For every triangulation $T \in \F(\Sigma)$ and for every multiarc $A$, we have $$d(T, \F_A) \leq i(T,A).$$ 
\end{corollary}

We remark that if $S$ is a triangulation, $\F_S=S$ and the path described in Theorem \ref{our-algo} is a path joining $T$ and $S$. As such we also have the following corollary.

\begin{corollary}
The flip graph $\F(\Sigma)$ is connected, and for any two triangulations $T, S \in \F(\Sigma)$ we have $d(T,S) \leq i(T,S).$
\end{corollary}

Our construction also enjoys the following properties.

\begin{corollary}\label{corollary}
The path from $T$ to $\F_A$ described in Theorem \ref{our-algo} has the following properties: 
\begin{enumerate}
\item If there exists $a \in A$ such that $a \in T_i$ then $a \in T_j$ for all $j \geq i$; 
\item If $t \in T_i$ is such that $i(t, A) = 0$ then $t \in T_j$ for all $j \geq i$. 
\end{enumerate}
\end{corollary}
We will see later that all the geodesic paths between a triangulation and the stratum of a multiarc also have these properties. We use this corollary to deduce the following. 

\begin{corollary}\label{strata-conn}
For every multiarc $A$, the stratum $\F_A$ is arcwise connected. 
\end{corollary}

\begin{proof}
Let $S,T \in \F_A$ be two triangulations, and consider the path from $T$ to $\F_S = S$ described in Theorem \ref{our-algo}. We remark that for all $a \in A$ we have $a \in T$ and $i(a, S) = 0$. By Corollary \ref{corollary}, for all $a \in A$ we have $a \in T_i$, so $T_i \in \F_A$ for all $i = 0, \ldots, n$. We conclude that the path described is contained in $\F_A$.
\end{proof}

\subsubsection{A lower bound on distances}	
As a complement to our upper bound on distance in terms of intersection number, we now show how intersection also provides a lower bound on distance. We begin with the following observation. 

\begin{lemma}\label{low:lemma-1}
Let $T$ and $T'$ be two triangulations in $\Sigma$ that differ by one flip, and let $A$ be a multiarc with $|A|$ components. 
Then we have: 
$$ i(T', A) \geq 2 \cdot \max_{t \in T} i(t, A) - 2 \cdot |A| .$$ 
\end{lemma}

\begin{proof}
Assume that $T$ and $T'$ differ by a flip on $t$, we set $T' = T \setminus \{t\} \cup \{t'\}$. Let $h\in T$ be an arc such that $i(h, A) = \max_{t\in T} i(t,A)$. 
We have: 
\begin{equation}\label{eq:2}
\begin{split} 
i(T',A) &= i(T,A) - i(t,A) + i(t',A) \geq i(T,A) - i(t,A) \\ 
&\geq i(T,A) - i(h,A) 
\end{split}
\end{equation}

By Lemma \ref{flip-convenient-0} the arc $h$ is flippable. In the notation of Lemma \ref{flip-convenient-1} , we have 
\begin{align}\label{eq:3}
\begin{split}
i(h, A) &= \alpha + \eta + \beta + \delta + \epsilon + \epsilon' + \epsilon'' \\ 
i(x, A) & = \alpha + \delta \\ 
i(y, A) &= \eta + \beta + \delta + \epsilon'' \\ 
i(w, A) & = \beta + \epsilon' \\ 
i(z, A) & = \alpha + \eta + \epsilon
\end{split}
\end{align}

Remark that $\epsilon + \epsilon' + \epsilon''' $ is the total number of arcs that terminates in the quadrilateral containing $h$, therefore $\epsilon + \epsilon' + \epsilon''' \leq 2|A|$. Combining Equations \ref{eq:3}, we have: 
\begin{equation}\label{eq:4}
\begin{split}
i(T', A) & \geq i(T, A) - i(h,A) \\
&\geq (i(x,A) + i(y,A) + i(w,A) + i(z,A) + i(h,A) ) - i(h,A) \\ 
& = (\alpha + \delta) + (\eta + \beta + \delta + \epsilon'') + (\beta + \epsilon' ) + (\alpha + \eta + \epsilon) \\
& = 	2 i(h, A) - (\epsilon + \epsilon' + \epsilon'') \\ 
& \geq 2 i(h,A) - 2 \cdot |A| 
\end{split}
\end{equation}
\end{proof}

\begin{theorem}
If $T$ is a triangulation and $A$ is a multiarc such that $\max_{t \in T} i(t, A) \geq 2|A|$, then 
$$d(T, \F_A) \geq \left \lfloor \frac{\log(i(T,A)) - \log( 2|A| - 1 )}{\log(\kappa)} \right \rfloor - 2.$$
\end{theorem}

\begin{proof}
Let $h \in T$ be an arc such that $i(h,A) = \max_{t \in T} i(t,A)$. 
We have: 
\begin{equation}\label{eq:1}
i(h, A) \geq \frac{i(T,A)}{\kappa}. 
\end{equation}

Let $T'$ be a triangulation that differs from $T$ by one flip. By Lemma \ref{low:lemma-1} the following holds when $i(h,A) \geq 2 |A|$ : 
\begin{equation}\label{eq:5}
\begin{split}
i(T',A) &\geq 2 i(h,A) - 2 \cdot |A| \\
&\geq i(h,A) \\
& \geq \frac{i(T,A)}{\kappa} .
\end{split}
\end{equation}
We note that the case $i(h,A) \leq 2 |A|$ is not very interesting because in this case $T$ is not too far from $\F_A$. Indeed, by Lemma \ref{dist-up} we have: $d(T,\F_A) \leq 2 \kappa \cdot i(h,A) \leq 2 \kappa \cdot |A|$. 

Let $d = d(T, \F_A) $, and let $T=T_0 \rightarrow \ldots \rightarrow T_d \in \F_A$ be a geodesic path from $T$ to $\F_A$.
Let $m \leq d$ be the smallest integer such that $\max_{t \in T_{m+1}} i(t,A) \leq 2 |A| -1$ and for every $j \leq m$ we have $\max_{t \in T_{j}} i(t, A) \geq 2 |A|$.
We have:
\begin{equation} 
\begin{split}
i(T_{m+1}, A) \leq (2|A| -1) \cdot \kappa.
\end{split}
\end{equation}
By Lemma \ref{eq:5} and the above remark, we also have: 
\begin{equation} 
\begin{split}
 i(T_{m+1}, A) \geq \frac{i(T_{m}, A)}{\kappa} \geq \ldots \geq \frac{i(T_0, A)}{\kappa^{m+1}}
\end{split}
\end{equation}

We have the following inequality that we solve for $m$: 
\begin{equation}
\begin{split}
(2|A| -1) \cdot \kappa &\geq \frac{i(T_0, A)}{\kappa^{m+1}} \\ 
\kappa^{m+2} &\geq \frac{i(T,A)}{2|A| - 1} \\ 
m & \geq \frac{\log i(T,A) - \log(2|A| - 1)}{\log(\kappa)} - 2. 
\end{split}
\end{equation}
We conclude: $$d(T, \F_A) = d \geq m \geq \frac{\log i(T,A) - \log(2|A| - 1)}{\log(\kappa)} -2.$$
\end{proof}

We remark that if $\max_{t \in T} i(t, A) \leq 2|A| - 1$ then by Lemma \ref{dist-up}

$$d(T, \F_A) \leq i(T,A) \leq (2|A|-1)\cdot \kappa.$$

\begin{corollary}
If $T$ and $S$ are two triangulations, then $$d(T,S) \geq \left \lfloor \frac{\log(i(T,S))}{\log(\kappa)} \right \rfloor - 4.$$ 
\end{corollary}

\subsection{Examples of the relationship between flip graphs and the mapping class group}

In this section, we provide two examples of how one can use the flip graph to study the mapping class group. They are completely independent from the rest of the paper but are provided to illustrate the variety of ways in which the quasi-isometry between the two objects can be used.

\subsubsection{Mapping tori and pseudo-Anosov homeomorphisms}

The following proposition follows from a standard construction in 3-dimensional topology known as the \emph{layered triangulation} of the mapping torus of a pseudo-Anosov homeomorphism. 
\begin{proposition}\label{Agol}
For every pseudo-Anosov $\phi \in \MOD(\Sigma)$ and for every triangulation $T \in \F(\Sigma)$ we have: 
$$d(T, \phi(T)) \geq \frac{ \mathrm{vol}(M_\phi)}{\pi}, $$ where $\mathrm{vol}(M_\phi)$ is the volume of the mapping torus $M_\phi$ of $\phi$.
\end{proposition}

\begin{proof}
Consider a geodesic path of flips $T \rightarrow \ldots \rightarrow \phi(T)$. The number of hyperbolic tetrahedra in the layered triangulation of $M_\phi$ associated to this path is equal to $d(T, \phi(T))$. For details on the layered triangulation of a mapping torus, we refer to \cite{Agol}. 
\end{proof}
Our second application is the following.
\begin{corollary}
For every pseudo-Anosov $\phi \in \MOD(\Sigma)$ the cyclic subgroup $\langle \phi \rangle$ is undistorted in $\MOD(\Sigma)$. 
\end{corollary}

\begin{proof}
We first prove that for every triangulation $T$ and for every $\phi$ we have: 
$$\frac{n \cdot \mathrm{vol}(M_\phi)}{\pi} \leq d(T, \phi^n (T)) \leq n \cdot d(T, \phi(T)). $$ 
	By Lemma \ref{mcg-flipgraph} this suffices to prove the corollary. 

The upper bound follows immediately by the triangle inequality. For the lower bound we use Proposition \ref{Agol}. We remark that since $M_{\phi^n}$ is a finite cover of degree $n$ of $M_\phi$ then
$$\mathrm{vol}(M_{\phi^n}) = n \cdot \mathrm{vol}(M_\phi).$$
It follows $$d(T, \phi^n(T)) \geq \frac{ \mathrm{vol}(M_{\phi^n})}{\pi} =\frac{ n \cdot \mathrm{vol}(M_\phi)}{\pi} .$$ 

\end{proof}

\subsubsection{The cone construction}
Fix a complete finite-area hyperbolic metric $M$ on $\Sigma$ and a homeomorphism $\varphi$ between $M$ and $\Sigma$. Let $P = \{ p_1, \ldots, p_n \}$ be the set of 
punctures of $\Sigma$. It is a classical result of Birman and Series that the set of all simple geodesics on $M$ is nowhere dense on $M$. We choose a point on 
the complement of the closure of all the simple geodesics of $M$ and consider its image by $\varphi$ on $\Sigma$. We denote this point $p_{n+1}$ (on both 
$\Sigma$ and $M$). We now set $P' = P \cup \{p_{n+1} \}$ and let $\Sigma'$ be the punctured surface $\Sigma$ with an extra marked point at $p_{n+1}$. This construction 
is known as \emph{puncturing} (see \cite{RS}). 
Let $T$ be a triangulation of $\Sigma$, denote by $\mathcal{G}_M(T)$ the unique $M$-geodesic representative in its isotopy class (it is an ideal triangulation as the marked 
points become punctures). Then $p_{n+1}$ is contained in a unique triangle of $\Sigma \setminus{\mathcal{G}_M(T)}$. We then {\it cone} the triangle in $p_{n+1}$: by this 
we mean add arcs between $p_{n+1}$ and the three vertices of the triangle to obtain a triangulation of $\Sigma'$ that we denote by $\widehat{T}$. We will also refer to the arcs going to $p_{n+1}$ as the \emph{cone} on $p_{n+1}$. In the following we will 
denote by $d$ the flip distance on $\F(\Sigma)$ and by $d'$ the flip distance on $\F(\Sigma')$.

\begin{lemma}\label{cone-lemma}
The cone map 
\begin{align*}
\mathrm{\cone}_{M}: \F(\Sigma) &\to \F(\Sigma') \\
T &\mapsto \widehat{T}
\end{align*}
is well-defined and 2-Lipschitz. 
\end{lemma}

\begin{proof}
If $T'$ is a triangulation of $\Sigma$ isotopic to $T$, then $\mathcal{G}_M(T') = \mathcal{G}_M(T)$, so $\widehat{T'} = \widehat{T}$. 
Figure \ref{flips-cone-1} shows that if two triangulations differ by one flip, their images by the cone map differ by at most 2 flips.

\begin{figure}[h]
\leavevmode \SetLabels
\endSetLabels
\begin{center}
\AffixLabels{\centerline{\epsfig{file =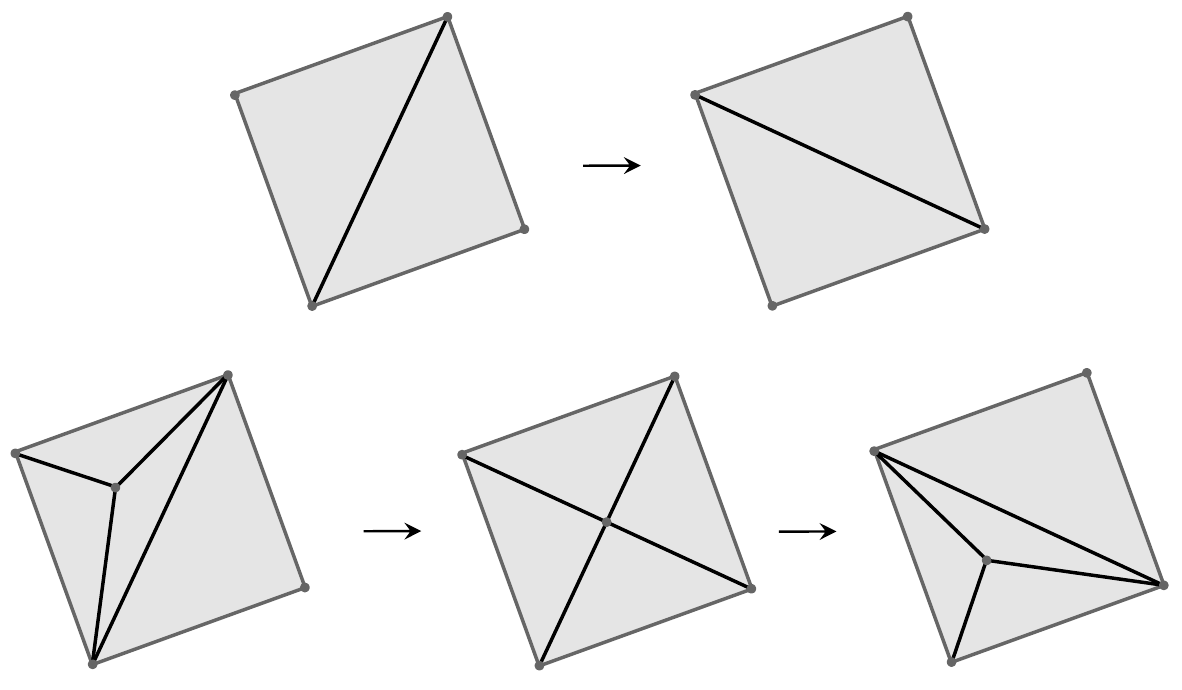,width=12.0cm,angle=0} }}
\vspace{-30pt}
\end{center}
\caption{A flip and a corresponding two flip move in the coned triangulation}\label{flips-cone-1}
\end{figure}

\end{proof}

We will now prove that $\cone_M$ is a quasi-isometric embedding. 
Fix a triangulation $H \in \F(\Sigma)$. Denote by 
$\omega_{H}: \MOD(\Sigma) \to \F(\Sigma)$ the orbit map of $H$ under $\MOD(\Sigma)$ as in Lemma \ref{mcg-flipgraph}. Similarly, denote by $\omega_{\widehat{H}}: \MOD(\Sigma') \to \F(\Sigma')$
the orbit map of $\widehat H$ under $\MOD(\Sigma')$. Recall that $\MOD(\Sigma)$ and $\MOD(\Sigma')$ are related by the Birman exact sequence, where the map $f: \MOD(\Sigma') 
\to \MOD(\Sigma)$ is the \emph{forgetful map}: 
$$ 1 \to \pi_1(\Sigma, p_{n+1}) \to \MOD(\Sigma') \overset{f} \to \MOD(\Sigma) \to 1 $$
 
\begin{lemma}\label{F-lipschitz}
Let $f: \MOD(\Sigma') \to \MOD(\Sigma)$ be the forgetful map.
Let $\omega^{-1}_{\widehat H}: \F(\Sigma') \to \MOD(\Sigma')$ be a quasi-inverse of $\omega_{\widehat{H}}$.
The following map is quasi-Lipschitz.
\begin{align*}
F: \F(\Sigma') & \to \F(\Sigma) \\
 T & \mapsto \omega_{H} \circ f \circ \omega^{-1}_{\widehat{H}}(T) 
\end{align*} 
For all $\psi' \in \MOD(\Sigma')$ we have $ F(\psi'( \widehat{H})) = f (\psi')(H)$. 
\end{lemma}

\begin{proof}
 It is immediate to see that $f$ is 1-Lipschitz with respect to the Humphreis generators of $\MOD(\Sigma)$. The assertion follows by composition with the two quasi-isometries.
\end{proof}
We remark that the quasi-Lipschitz constants of $F$ depend on the diameter of $\MF(\Sigma)$ and a choice of generators for $\MOD(\Sigma)$ and $\MOD(\Sigma') $. 
\begin{lemma}\label{2N}
For every $\psi \in \MOD(\Sigma)$ there exists $\phi' \in f^{-1}(\psi) \subset \MOD(\Sigma')$ such that
$$d'( \widehat{\psi H}, \phi' (\widehat{H})) \leq 2 \tilde{\kappa} $$
\end{lemma}

\begin{proof}
Fix $\psi \in \MOD(\Sigma)$ and choose $\psi' \in \MOD(\Sigma')$ such that $f (\psi')= \psi$, that is, $\psi$ and $\psi'$ are homeomorphisms of $\Sigma$ isotopic rel $P$ . 
Let us first compare $\widehat{\psi (H)}$ and $\psi'(\widehat{H})$. 

To construct $\psi'(\widehat{H})$ we proceed as follows. Set $\Sigma \setminus \mathcal{G}_M(H) = \bigcup \Delta_i$ where $\Delta_i$ is a triangle. We assume $p_{n+1} \in \Delta_1$, so that $\hat{H}$ is obtained by $H$ coning $\Delta_1$ and $\phi'(\widehat H)$ is obtained coning $\psi'(\Delta_1)$. To construct $\widehat{\psi(H)}$ we proceed as follows. Set $\Sigma \setminus \mathcal{G}_M(\psi(H)) = \bigcup \Delta_i'$, where $\Delta'_i$ is a 
triangle. Since $\psi' \in f^{-1}(\psi)$, we can assume (up to reordering) that $\Delta_i' $ is isotopic to $\psi'(\Delta_i)$ relative to $P$. 
We have two cases: 
\begin{enumerate}
\item $p_{n+1} \in \Delta_1'$; 
\item $p_{n+1} \not \in \Delta_1'$. 
\end{enumerate}
In case (1), we can glue the homeomorphisms $\Delta_i' \to \psi'( \Delta_i)$ in order to construct a homeomorphism $\theta: \Sigma \to \Sigma $ that also fixes $p_{n+1}$. By construction, $\theta$ is an element of $\MOD(\Sigma')$ isotopic to the identity rel $P$, that is, $\theta$ belongs to the kernel of the forgetful map $f$, and we obtain $\theta(\widehat{\psi H}) = \psi'( \widehat{H})$. Consider the mapping class $\phi' = \theta^{-1} \circ \psi' \in \MOD(\Sigma')$, by construction $$f(\theta^{-1} \circ \psi') = f(\psi') = \psi \mbox{ and } \widehat{\psi H} = \phi'( \widehat{H}),$$ and we are done.

In case (2), assume $p_{n+1} \in \Delta_j'$ with $j \neq 1$. We will now see that using at most $2 \tilde{\kappa}$ flips we can move the cone on $p_{n+1}$ inside a triangle isotopic to $\Delta_1'$. More precisely, a sequence of two flips as in Figure \ref{flips-cone-2} moves the cone in a triangle adjacent to $\Delta_j'$. Note that this sequence of flips does not change the isotopy class relative to $P$ of the arcs not connected to $p_{n+1}$. The final triangulation $T_1$ we obtain has the following properties: 
\begin{itemize}
\item $T_1$ has a cone on $p_{n+1}$; 
\item $T_1$ agrees with $\widehat{\psi H}$ outside the quadrilateral in Figure \ref{flips-cone-2}; 
\item the arcs of $\widehat{\psi H}$ and $T_1$ that are not connected to $p_{n+1}$ are pairwise isotopic relative $P$. 
\end{itemize}
If the triangle of $T_1$ containing the cone on $p_{n+1}$ is isotopic to $\Delta_1'$ relative to $P$, then we can proceed as in case (1). Indeed, we construct an homeomorphism $\theta: \Sigma' \to \Sigma'$ such that $\theta(T_1) = \psi'(\widehat{H})$ and $\theta$ is isotopic to the identity relative to $P$. We then set $\phi'= \theta^{-1} \circ \psi'$, and we have $T_1 = \phi'(\widehat{H})$.
Otherwise, if the triangle of $T_1$ containing $p_{n+1}$ is not isotopic to $\Delta_1'$, we keep on performing sequences of flips like in Figure \ref{flips-cone-2} in order to move the cone on $p_{n+1}$. After at most $\tilde{\kappa}$ sequences of flips, we get to a triangulation $T_{\tilde{\kappa}}$ whose cone on $p_{n+1}$ lies inside a triangle isotopic to $\Delta_1'$. Arguing as above, we get a homeomorphism $\theta: \Sigma' \to \Sigma'$, isotopic relative to $P$ to the identity, such that $\theta(T_{\tilde{\kappa}}) = \psi'(\widehat{H})$. We then set $\phi'= \theta^{-1} \circ \psi'$, we have $T_{\tilde{\kappa}} = \phi'(\widehat{H})$. 
We conclude as follows: 
\begin{align*}
d'( \widehat{\psi (H)}, \phi'( \widehat{H})) & \leq d'( \widehat{\psi (H)}, T_{\tilde{\kappa}}) + d'( T_{\tilde{\kappa}} , \phi'( \widehat{H})) \\
& \leq 2 \cdot \tilde{\kappa}. 
\end{align*} 

\begin{figure}[h]
\leavevmode \SetLabels
\endSetLabels
\begin{center}
\AffixLabels{\centerline{\epsfig{file =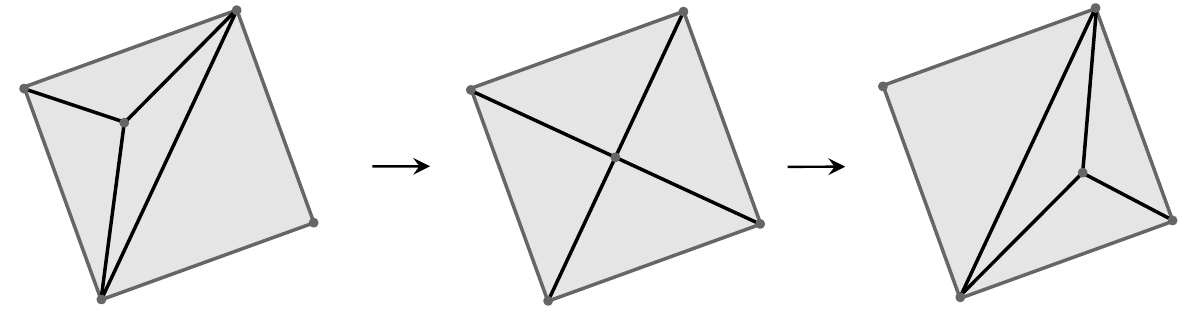,width=12.0cm,angle=0} }}
\vspace{-30pt}
\end{center}
\caption{Passing from $\widehat{\psi H}$ to $T_1$}\label{flips-cone-2}
\end{figure}
\end{proof}

\begin{theorem}
$\mathrm{cone}_M: \F(\Sigma) \to \F(\Sigma')$ is a quasi-isometric embedding.
\end{theorem}

\begin{proof}
By Lemma \ref{cone-lemma}, $\cone_M$ is 2-Lipschitz. We will now prove that there exists some universal constants $A', B'>0$ such that for any $S,T \in \F(\Sigma)$ we have $d(S, T) \leq A' \cdot d'(\widehat{S}, \widehat{T}) + B'$. 
Set $R = \diam \MF(\Sigma)$. Since the orbit of $H$ under $\MOD(\Sigma)$ is $R$-dense, we can find $\psi, \phi \in \MOD(\Sigma)$ such that 
\begin{equation} \label{eq:R}
d(S, \psi H) \leq R ~~ \mbox{ and } ~~ d(T, \phi H) \leq R
\end{equation}
By Lemma \ref{cone-lemma} we have:
\begin{equation} \label{eq:2R}
d'(\widehat{S}, \widehat{\psi H}) \leq 2R~~ \mbox{ and } ~~ d'(\widehat{T}, \widehat{\phi H}) \leq 2R
\end{equation}
Finally fix $\psi', \phi' \in \MOD(\Sigma')$ as in Lemma \ref{2N}. We use the notation $a \prec b$ to mean $a \leq k b + h$ for constants $k$ and $h$. 

With this notation: 
\begin{align*}
d(S, T) & \leq d(\psi H, \phi H) + 2R && \mbox{by \ref{eq:R}} \\ 
	 & = d( F(\psi'( \widehat H)), F(\phi' (\widehat H))) + 2R \\ 
 & \prec d'(\psi' (\widehat H) , \phi'( \widehat H)) && \mbox{ by Remark \ref{F-lipschitz}} \\
 & \prec d'(\psi'( \widehat H), \widehat{\psi H}) + d'(\widehat{\psi H}, \widehat{\phi H}) + d'(\phi'( \widehat H), \widehat{\phi H}) \\ 
 & \prec d'(\widehat{\psi H}, \widehat{\phi H}) + 4\cdot \tilde{\kappa} && \mbox{by Lemma \ref{2N} } \\
 & \prec d'(\widehat{\psi H} , \widehat{S} ) + d'(\widehat{S} , \widehat{T} ) + d'(\widehat{\phi H} , \widehat{T} ) \\
 & \prec d'(\widehat{S} , \widehat{T} ) + 4R && \mbox{ by Equation \ref{eq:2R}} 
\end{align*}
For $F$ is $(A,B)$ quasi-Lipschitz, and if we keep track of the constants in the above inequalities, we have:
$$d(S, T) \leq A' \cdot d'(\widehat{S} , \widehat{T} ) + B', $$ where $A' = A$ and $B' = 4 A R + 4 A \tilde{\kappa} + B + 2R$. 
\end{proof}

The following result was already proved by Mosher \cite{Mosher2} using a different method, and stated by Rafi-Schleimer \cite{RS} using the marking graph as a large scale model for $\MOD(\Sigma)$. 
\begin{corollary}
There is a quasi-isometric embedding $\MOD(\Sigma) \hookrightarrow \MOD(\Sigma')$.
\end{corollary}

\begin{proof}
Consider the following commutative diagram:
 $$\xymatrix{
\F(\Sigma) \ar[r]^{\mathrm{cone}_M} &\F(\Sigma') \\
\MOD(\Sigma) \ar[u]^{{\omega_H}_|} \ar[r] & \MOD(\Sigma') \ar[u]_{\omega_{\widehat{H}}} }
$$
By Lemma \ref{mcg-flipgraph} both $\omega_H$ and $\omega_{\widehat{H}}$ are quasi-isometries. The assertion then follows from the above theorem. 
\end{proof}

 \section{Convexity of strata and applications}
As we saw previously, for any multiarc $A$ the stratum $\F_A$ is connected. We denote by $d_A$ the shortest path distance on $\F_A$. In this section we prove that the natural inclusion $(\F_A, d_A) \hookrightarrow (\F,d)$ is an isometric embedding. Furthermore, we prove that $\F_A$ is strongly convex in $\F(\Sigma)$. The main ingredient in our proof is a 1-Lipschitz retraction of $\F(\Sigma)$ on $\F_A$. 

\subsection{The projection theorem} 
Let $a$ and $t$ be two arcs. Choose an orientation on $a$, denote by $a^+$ the oriented arc, and let $\mathrm{push}_{a^+}(t)$ be the multiarc defined as follows: 
\begin{itemize}
 \item if $i(a, t) =0$ then $\mathrm{push}_{a^+}(t) = t$; 
 \item if $i(a, t) \neq 0$ then $\mathrm{push}_{a^+}(t)$ is the multiarc obtained by ``combing'' $t$ following the orientation 
of $a$ as in Figure \ref{fig:combing}. Each arc in $\mathrm{push}_{a^+}(t)$ (provided $i(a, t) \neq 0$) has at least one endpoint that coincides with the final endpoint of $a$. \end{itemize}

\begin{figure}[htbp]
\begin{center}
\includegraphics[width=9cm]{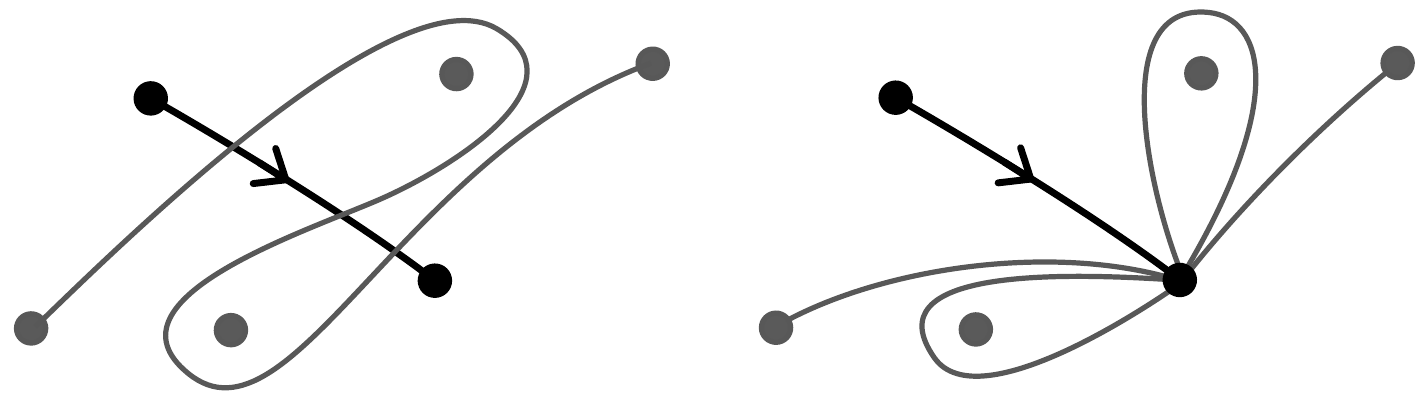}
\caption{Combing along an oriented arc}\label{fig:combing}
\end{center}
\end{figure}

The following lemma follows immediately by the above construction.
\begin{lemma}\label{geo-strata:lemma}
If $s$ and $t$ are arcs, then $ i(\mathrm{push}_{a^+}(s), \mathrm{push}_{a^+}(t)) \leq i(s,t)$. 
\end{lemma}

If $T=(t_1, \ldots, t_\kappa)$ is a triangulation of $\Sigma$, we denote by $\mathrm{push}_{a^+}(T)$ the multiarc obtained collecting the isotopy classes of all the arcs $\mathrm{push}_{a^+}(t_i)$: $\mathrm{push}_{a^+}(T) = [\mathrm{push}_{a^+}(t_1), \ldots , \mathrm{push}_{a^+}(t_\kappa)]$.
We remark that the set $\{\mathrm{push}_{a^+}(t_i)\}$ may contain isotopic arcs. 

\begin{lemma}\label{projection1}
The map 
\begin{align*}
 \pi_{a^+} : \F(\Sigma) & \to \F_a \\
 T & \mapsto (\mathrm{push}_{a^+}(T), a)
 \end{align*}
 is a 1-Lipschitz retraction on $(\F_a, d_a)$.
\end{lemma}

\begin{proof}
We first prove that $\pi_{a^+}(T)$ is also a triangulation of $\Sigma$. 
If $a$ is one of the arcs in $T$, the assertion is trivial. Suppose $a\cap T \neq \emptyset$.

We consider a parametrization of $a:[0,1]\to \Sigma$ following the orientation of $a^{+}$. We suppose that $a$ intersects $T$ minimally and all intersections are transversal so we denote
$$\tau_0=0<\tau_1<\hdots <\tau_{N+1}=1$$
the values of $\tau$ for which $a(\tau) \in T$. Note that $N = i(a,T)$. 

For each $\tau' \in [0,1]$, we consider the following decomposition $D_{\tau'}$ of $\Sigma$ constructed as follows. To begin, $D_{\tau'}$ contains all arcs of $T$ that do not cross $a|_{\tau=0}^{\tau'}$, contains all vertices of $T$ and has one extra vertex $a(\tau')$. Furthermore, it also contains the arc $a|_{\tau=0}^{\tau'}$. We add arcs iteratively as follows. For each $\tau_i$, $i=1,\hdots, N-2$, the point $a(\tau_i)$ will cut a preexisting arc, say $b_i$, into two subarcs $b'_{i}$ and $b''_{i}$. We add these to the decomposition and they continue to belong to the decomposition for $\tau'>\tau_i$ by concatenating them with the arc $a|_{\tau=\tau_i}^{\tau'}$ in the obvious way. At parameter $\tau'$ we denote the resulting arcs $b'_{i}(\tau')$ and $b''_{i}(\tau')$. $D_{\tau'}$ is the union of all these arcs up to isotopy fixing the vertices (so any isotopy class is only counted once). 

We want to show that $D_{\tau_{N+1}}$ is a triangulation of $\Sigma$ with the same vertex set as $T$. Before showing this we claim that for $0\leq i<N+1$, $D_{\tau_i}$ is a set of arcs decomposing $\Sigma$ into triangles and into one quadrilateral which is simply a triangle with an additional vertex $a(\tau_i)$.

We prove our claim by analyzing the decomposition as $\tau$ varies. The key point is that the decomposition only changes for the values $\tau_i$.

For $i=1$ as all we have added is an arc and a point that splits the first triangle traversed by $a$ into two triangles. As we have added a vertex in $a(\tau_1)$, the following triangle of $T$ traversed by $a$ is now a quadrilateral (see Figure \ref{fig:firststep}). 

\begin{figure}[h]
\leavevmode \SetLabels
\L(.56*.97) $a(0)$\\
\L(.47*.41) $a(\tau_1)$\\
\endSetLabels
\begin{center}
\AffixLabels{\centerline{\epsfig{file =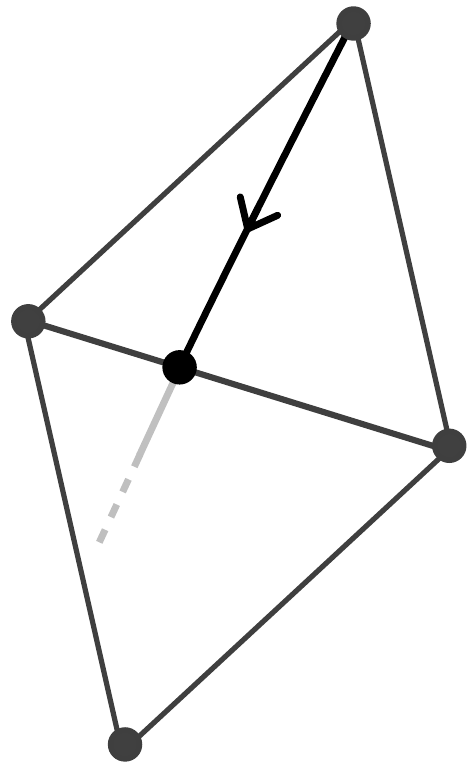,width=3.0cm,angle=0} }}
\vspace{-30pt}
\end{center}
\caption{When the second triangle becomes a quadrilateral} \label{fig:firststep}
\end{figure}

Now suppose by induction that at parameter $\tau_i$ with $i<N$ the decomposition $D_{\tau_i}$ is as claimed and we now analyze $D_{\tau_{i+1}}$. 

To obtain $D_{\tau_{i+1}}$ from $D_{\tau_i}$ we have a continuous family $D_\tau$ with $\tau \in [\tau_{i},\tau_{i+1}]$. Note that $a|_{\tau=\tau_i}^{\tau_{i+1}}$ is a simple path crossing the only quadrilateral of $\Sigma \setminus D_{\tau_i}$. While $\tau' \in [\tau_{i},\tau_{i+1}[$, $D_\tau$ is just obtained by pushing $a(\tau)$ along the path $a|_{\tau=\tau_i}^{\tau'}$ and thus (up to homeomorphism) is a carbon copy of $D_{\tau_i}$. 

\begin{figure}[h]
\leavevmode \SetLabels
\L(.66*.41) $a(\tau_{i+1})$\\
\L(.29*.73) $a(\tau_i)$\\
\endSetLabels
\begin{center}
\AffixLabels{\centerline{\epsfig{file =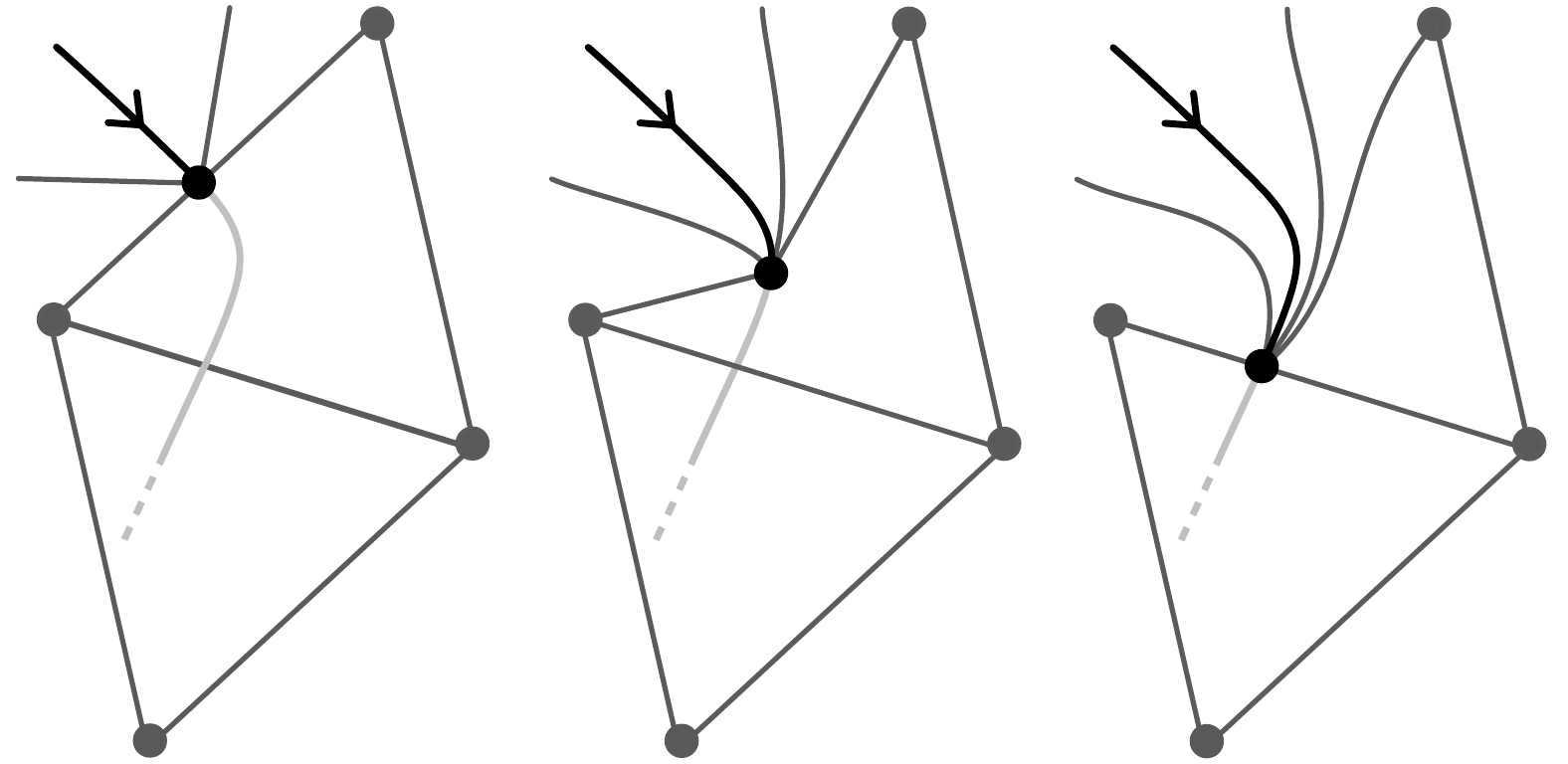,width=9.0cm,angle=0} }}
\vspace{-30pt}
\end{center}
\caption{The general step} \label{fig:inductionstep}
\end{figure}

We need to analyse what happens at $\tau = \tau_{i+1}$. Two of the arcs of the quadrilateral become one and as such the quadrilateral collapses to a triangle. More precisely, the point $a(\tau_{i+1})$ lies on an arc of $D_{\tau_i}$ so divides this arc into two arcs in $D_{\tau_{i+1}}$; adding this vertex turns the ``next" triangle into a quadrilateral. This proves the general step. The above process is illustrated in Figure \ref{fig:inductionstep}.

What remains to be seen is the final step, when $\tau\in [\tau_N,\tau_{N+1}]$. This final step is very similar to what happens before with the notable difference that the point $a(\tau_{N+1})$ was already a vertex of the decompositions $D_\tau$. So instead of splitting a previous arc into two parts, the quadrilateral containing $a(\tau)$ for $\tau\in [\tau_N,\tau_{N+1}[$ collapses completely leaving only triangles in $D_{\tau_{N+1}}$ (see Figure \ref{fig:finalstep}).

\begin{figure}[h]
\leavevmode \SetLabels
\L(.825*.06) $a(\tau_{N+1})$\\
\L(.285*.50) $a(\tau_N)$\\
\endSetLabels
\begin{center}
\AffixLabels{\centerline{\epsfig{file =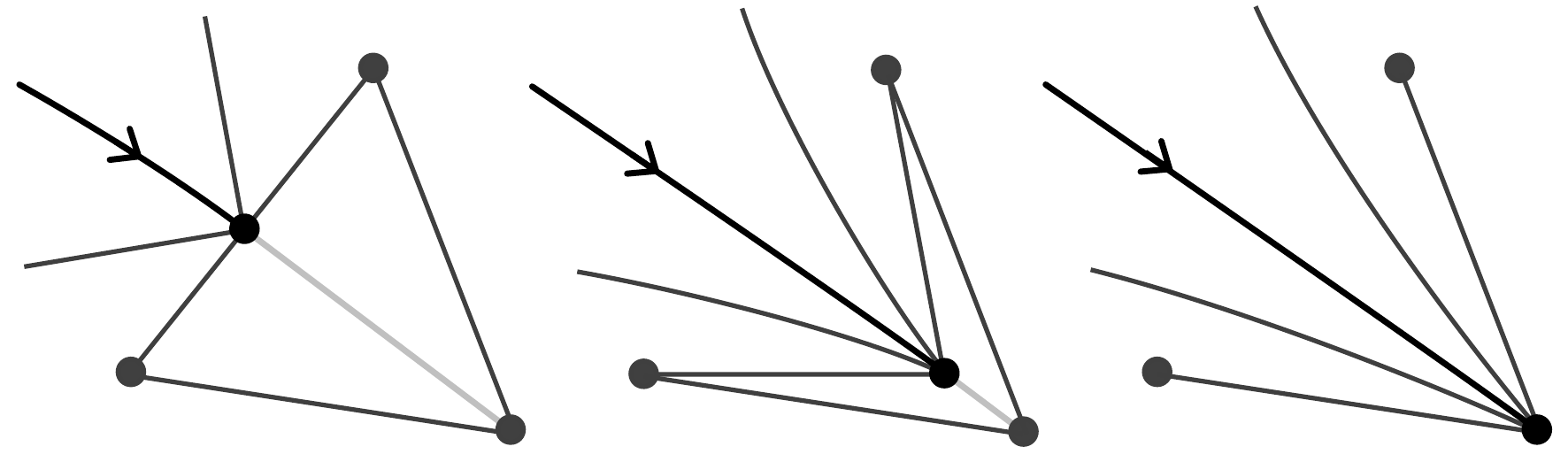,width=10.0cm,angle=0} }}
\vspace{-30pt}
\end{center}
\caption{The final step} \label{fig:finalstep}
\end{figure}

This concludes the proof that $\pi_{a^+}(T)$ is a triangulation.

It is straightforward to see that $\pi_{a^+}$ is a retraction. In fact, the restriction of $\pi_{a^+}$ to $\F_a$ is the identity and $\pi_{a^+}$ is onto by construction. 

Let us now prove that $\pi_{a^+}$ is 1-Lipschitz. Recall that $T_1, T_2 \in \F(\Sigma)$ differ by a flip if and only if $ i(T_1, T_2)=1$. By Lemma \ref{geo-strata:lemma} we have $ i(\pi_{a^+}(T_1),\pi_{a^+}(T_2))\leq 1 $. We deduce that either $\pi_{a^+}(T_1)$ and $\pi_{a^+}(T_2)$ also differ by a flip or they coincide. 
Let $T$ and $T'$ be two vertices in $\F(\Sigma)$ and $\gamma: T=T_0 \ldots T_m=T'$ is a geodesic path in $\F(\Sigma)$ joining them. 
By the above argument, $\pi_{a^+}(\gamma): \pi_{a^+}(T) \ldots \pi_{a^+}(T')$ is a 
path in $\F_a$ of length at most $m$, so $d(\pi_{a^+}(T), \pi_{a^+}(T')) \leq d(T,T')$ and $\pi_{a^+}$ is 1-Lipschitz.
\end{proof}

\begin{lemma}\label{pre-convexity}
Let $A$ be a multiarc and $T$ be a triangulation. If there exists $t \in T$ such that $ i(t,A) = 0$ then every geodesic path from $T$ to $\F_A$ 
is contained in $\F_t$.
\end{lemma}

\begin{proof}
Let $\gamma: T=T_0 \ldots T_n$ be a shortest path from $T_0$ to $\F_A$. We shall prove that for all $i$, $T_i \in \F_t$.
We begin by choosing an orientation on $t$. Observe that $p_{t^+}(T_n) \in \F_A$ and $T_0 = \pi_{t^+}(T_0)$ by construction, so $\pi_{t^+}(\gamma)$ is also a path from $T_0$ to $\F_A$.
We now argue by contradiction. Let $i \geq 0$ the smallest integer such that $t \in T_i$ and $t \not \in T_{i+1}$ (that is, the arc $t$ is flipped). Necessarily we have $i(t, T_{i+1}) = 1$ and by construction 
$$\pi_{t^+}(T_i) = \pi_{t^+}(T_{i+1}) = T_i$$
so the length of $\pi_{t^+}(\gamma)$ is at most $n-1$. This implies that $\pi_{t^+}(\gamma)$ is shorter than $\gamma$, in contradiction with the assumption that $\gamma$ is geodesic. 
\end{proof}

\begin{theorem}\label{convexity-1}
For every arc $a$, the stratum $\F_a$ is strongly convex. 
\end{theorem}

\begin{proof}
Let $T_0$ and $T_m$ be two vertices in $\F_a$ and let $\gamma: T_0 \ldots T_m$ be a geodesic path in $\F(\Sigma)$ joining them. By Lemma \ref{projection1} $\pi_{a^+}(\gamma) $ is a path in $\F_a$ with endpoints $T_0$ and $T_m$ and we have $d_a(T_0, T_m) \leq m = d(T_0,T_m) $. It follows that the inclusion $\F_a \hookrightarrow \F(\Sigma)$ is an isometric embedding. 
The strong convexity of $\F_a$ follows from Lemma \ref{pre-convexity} with $A = a$ and $t=a$: for all $i=0, \ldots, m$ we have $T_i \in \F_a$. 
\end{proof}

\begin{theorem}\label{projection2}
Let $A^\sigma= (a_1^+, \ldots, a_m^+)$ be a multiarc whose $m$ components are enumerated and oriented. 
The map $\pi_{A^\sigma} = \pi_{a_m^+} \circ \ldots \circ \pi_{a_1^+}: \F(\Sigma) \to \F_A$ is well-defined and a 1-Lipschitz retraction.
\end{theorem}

\begin{proof}
Since the arcs in $A$ are all disjoint, by Lemma \ref{pre-convexity} we have
$$\pi_{A^\sigma}(\F(\Sigma)) = \F_{a_1} \cap \ldots \cap \F_{a_m} = \F_A.$$
By Lemma \ref{projection1} the map $\pi_{A^\sigma}$ is 1-Lipschitz and a retraction. 
\end{proof} 

We remark that the map $\pi_{A^\sigma}$ does depend on the choice of the orientation and enumeration of the arcs in $A$. We will study this dependence later. 

\begin{theorem}\label{convexity-2}
For every multiarc $A$, the stratum $\F_A$ is strongly convex. 
\end{theorem}

\begin{proof}
This is a direct corollary of Theorem \ref{convexity-1}. Note that $\F_A = \bigcap_{a \in A} \F_a$ and the intersection of strongly convex subspaces is strongly convex.

\end{proof}

\subsection{Applications}\label{ss:applications}

We now focus on some applications of the above results and in particular of Theorem \ref{convexity-2}. 

\subsubsection{Projections and distances}

We begin by looking at some immediate consequences on distances and projection distances to strata. 

The following proposition is essentially the definition of distance on $\F_A$ combined with Theorem \ref{convexity-2}. 

\begin{proposition}\label{sep-multiarc}
 Assume that $A$ is a multiarc such that ${\Sigma \setminus A} = \bigcup_{i=1}^h \Sigma_i$ where $\Sigma_i$ is a connected surface with boundary. Denote by $d_i$ the distance on $\F(\Sigma_i)$. For every $T \in \F_A$ denote by $T_i$ 
 the triangulation of $\Sigma_i$ induced by $T$. Then the map 
 \begin{align*} 
 \F_A &\longrightarrow ~\F(\Sigma_1) \times \ldots \times \F(\Sigma_h) \\
 T & \mapsto (T_1 , \ldots, T_h) 
 \end{align*}
is an isometry between $(\F_A, d)$ and $(\F(\Sigma_1) \times \ldots \times \F(\Sigma_h)~,~ d_1 + \ldots + d_h )$ 
\end{proposition}

\begin{proof}
By definition of $\F_A$, the map is an isometry from $(\F_A, d_A)$. By Theorem \ref{projection2} $d = d_A$ and the assertion follows. 
\end{proof}

\begin{proposition}
Let $A$ be a multiarc. For every choice $\sigma$ of enumeration and orientation of the arcs in $A$, we have: 
$d(\pi_{A^\sigma}(T), \pi_{A^\sigma}(S)) \leq i(T,S)$. 
\end{proposition}

\begin{proof}
It is a straightforward application of Lemmas \ref{dist-up} and \ref{geo-strata:lemma}.
\end{proof}

\begin{proposition}\label{sigma}
Let $A$ be a multiarc. For every choice $\sigma$ of enumeration and orientation of the arcs in $A$, we have
$d(T, \F_A) \leq d(T, \pi_{A^\sigma}(T)) \leq 2\cdot d(T, \F_A) .$
\end{proposition}

\begin{proof}
Let $S$ be a triangulation in $\F_A$ at minimal distance from $T$, so that $d(T,S) = d(T,\F_A)$. By Theorem \ref{projection2} $\pi_{A^\sigma}(S)=S$, it follows: 
\begin{align*}
d(T, \pi_{A^\sigma}(T)) &\leq d(T,S) + d(S, \pi_{A^\sigma}(T)) \\ 
& \leq d(T,S) + d(\pi_{A^\sigma}(S), \pi_{A^\sigma}(T)) \\
& \leq d(T,S) + d(S,T) \\ 
& = 2 d(T, \F_A).
\end{align*}
\end{proof}

\begin{corollary}
Let $A$ be a multiarc. For every choice $\sigma, \epsilon$ of enumeration and orientations of the arcs in $A$, we have $d(\pi_{A^\sigma}(T), \pi_{A^\epsilon}(T)) \leq d(T, \F_A)$.
\end{corollary}

\begin{proof}
 It follows immediately by Proposition \ref{sigma}. 
\end{proof}

The next consequence will use a result by Aramayona, Koberda and the second author about simplicial maps between flip graphs. To state the result we require the following notation: we say that a surface $\Sigma$ is exceptional if it is an essential subsurface of (and possibly equal to) a torus with at most two marked points, or a sphere with at most four marked points. In \cite{AKP}, it is proved that, for surfaces $\Sigma, \Sigma'$ with $\Sigma$ non-exceptional, all injective simplicial maps
$$
\phi: \F(\Sigma) \to \F(\Sigma')
$$
come from embeddings $\Sigma\to \Sigma'$ (that is $\Sigma$ is homeomorphic to a subsurface of $\Sigma'$). Note that it's obvious that you can construct simplicial maps this way; what's more surprising is that this is, provided your base surface is complicated enough, the only way such maps appear. Together with Theorem \ref{convexity-2}, the following is then immediate. 

\begin{corollary}
Suppose $\Sigma$ is non-exceptional, and let $\F(\Sigma) \to \F(\Sigma')$ be an injective simplicial map. Then $\F(\Sigma)$ is strongly convex inside of $ \F(\Sigma')$.
\end{corollary}

\subsubsection{On the large scale geometry of the mapping class group}

We now turn our attention to the large scale geometry of the mapping class group. 

\begin{lemma}\label{stabilizers}
Let $A$ be a multiarc and $\Stab (A)$ be the subgroup of $\MOD(\Sigma)$ that fixes the isotopy class of each arc in $A$. Then $\Stab (A)$ has a finite index subgroup isomorphic to $\MOD(\Sigma \setminus A)$. 
\end{lemma}

\begin{proof}
Assume that $A$ has $m$ connected components. Fix an orientation on each arc of $A$. It is immediate to see that the subgroup $\Stab^+ (A) < \Stab(A)$ consisting of the mapping classes that also fix the orientation of every arc in $A$ is isomorphic to 
$\MOD( \Sigma \setminus A)$, that is, the subgroup of the surface obtained cutting $\Sigma$ along $A$. The assertion follows from the exactness of the short sequence: 
$$ 1 \to \Stab^+(A) \to \Stab(A) \to \mathbb Z_2^m \to 1 .$$ 
\end{proof}

We can now prove the following. 

\begin{theorem}\label{proper}
For every vertex $T \in \F_A$, there is a commutative diagram: 
$$\xymatrix{
\F_A \ar@{^{(}->}[r] &\F(\Sigma) \\
\mathrm{Stab}(A) \ar[u]^{{\omega_T}_|} \ar@{^{(}->}[r] & \MOD(\Sigma) \ar[u]_{\omega_T} }
$$
where the inclusion $\F_A \hookrightarrow \F(\Sigma)$ is an isometry and the orbit map $\omega_T: \MOD(\Sigma) \to \F(\Sigma)$ restricts to a quasi-isometry ${\omega_T}_|: \Stab(A) \to \F_A $. 
Moreover, the inclusion $\mathrm{Stab}(A) \hookrightarrow \MOD(\Sigma)$ is a quasi-isometric embedding. 
\end{theorem}

\begin{proof}
The inclusion $\F_A \hookrightarrow \F(\Sigma)$ is an isometry by Theorem \ref{convexity-2}. 
By Proposition \ref{sep-multiarc} $\F_A$ is isomorphic and isometric to $\F(\Sigma\setminus A)$. Since the action of $\MOD(\Sigma\setminus A)$ on $\F(\Sigma\setminus A)$ is cocompact, so it is the action of $\Stab(A)$ on $\F_A$ by Lemma \ref{stabilizers}. By the \v{S}varc-Milnor lemma the orbit map $ \MOD(\Sigma) \ni \psi \mapsto \psi T \in \F_A$ is a quasi-isometry. By composition the inclusion $\Stab(A) \hookrightarrow \MOD(\Sigma)$ is a quasi-isometric embedding.
\end{proof}

\section{The diameters of the modular flip graphs}

The goal of this section is to prove upper and lower bounds on the diameters of modular flip graphs in terms of the topology of the surface (namely Theorem \ref{thm:diameters} from the introduction). 

Let $\Sigma$ be a surface of genus $g$ with $n$ labelled points. We assume $g\geq 1$ and $n \geq 2$ (for the case $n=1$ see Theorem \ref{thm:onepunctureupper} and Remark \ref{rem:onepuncture}, for the case $g=0$ see Theorem \ref{thm:genuszeroupper}). The case where the points are unlabelled is slightly easier and it will also be treated separately - see Remark \ref{rem:unmarked}.

We begin with a general observation which allows us to break bounds on $\MF(\Sigma)$ into different parts. The idea is to work with the punctures on one side and genus on the other. To do this we consider triangulations that contain an arc which separates the genus from the punctures: more precisely an arc $a$ which forms a loop based in a puncture and such that $\Sigma \setminus a = \Omega \cup \Gamma$ where $\Omega$ is a disk with $n-1$ punctures and one labelled point on the boundary and $\Gamma$ is of genus $g$ with a boundary component with a single marked point. Such a loop we call {\it puncture separating}.

For any choice of puncture on $\Sigma$, it is clear that (infinitely many) such loops based in this point exist but up to homeomorphism there is only one such loop (see Figure \ref{fig:cutpunctures}). 

\begin{figure}[h]
\leavevmode \SetLabels
\endSetLabels
\begin{center}
\AffixLabels{\centerline{\epsfig{file =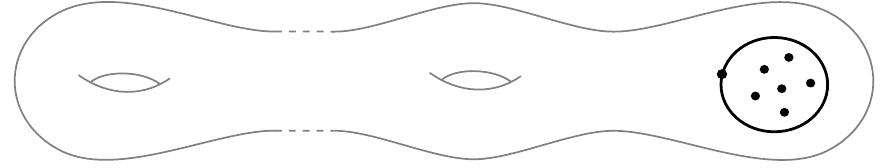,width=10.0cm,angle=0} }}
\vspace{-18pt}
\end{center}
\caption{A puncture cutting loop} \label{fig:cutpunctures}
\end{figure}

From this we can make the following observation: any two triangulations which are distinct up homeomorphism and both contain a puncture separating loop must be either distinct on $\Gamma$ or $\Omega$. As such:
\begin{equation}
\card\left( \MF(\Sigma) \right) > \card(\MF(\Gamma))\,\, \card(\MF(\Omega)).
\end{equation}\label{eq:card}
We will use this for our lower bounds in Section \ref{ss:lower}. 

For our upper bounds the following lemma will allow us to introduce a puncture separating arc in a minimal amount of flips. 

\begin{lemma}\label{lem:separatearc}
For any $T\in \MF(\Sigma)$ and any marked point $p$ of $\Sigma$, there exists a puncture separating loop $a$ based in $p$ with 
$$
i(a, T) \leq 2 (\kappa - n + 1)
$$
\end{lemma}

\begin{proof}
We think of $T$ as a graph embedded on $\Sigma$ and consider a spanning tree of this graph. A regular neighborhood of this tree is a simple closed curve $\gamma$ which satisfies 
$$
i(\gamma,T) \leq 2 ( \kappa - (n-1)) 
$$
as it intersects only half edges that do not belong to the tree and the tree has $n-1$ edges - see Figure \ref{fig:curve}. (The above inequality is in fact an equality but it is the inequality that we need.) 

\begin{figure}[h]
\leavevmode \SetLabels
\endSetLabels
\begin{center}
\AffixLabels{\centerline{\epsfig{file =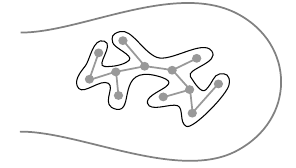,width=8.0cm,angle=0} }}
\vspace{-30pt}
\end{center}
\caption{The curve $\Gamma$} \label{fig:curve}
\end{figure}

From $\gamma$ and given a marked point $p$, we shall construct an arc as follows: as $\gamma$ surrounds all punctures, it must pass through a triangle that has $p$ as a vertex. We consider a simple arc $c$ in the triangle between $\gamma$ and $p$. Choosing an orientation on $c$ and $\gamma$, the concatenation of $c \gamma c^{-1}$ gives an isotopy class of arc which is the arc $a$ we are looking for. Notice that by construction it intersects $T$ in at most as many points as $\gamma$ and we have 
$$
i(a, T) \leq 2 ( \kappa - n +1) 
$$
as desired.

\begin{figure}[h]
\leavevmode \SetLabels
\L(.65*.85) $\gamma$\\
\L(.474*.61) $\tiny{c}$\\
\L(.65*.57) $a$\\
\endSetLabels
\begin{center}
\AffixLabels{\centerline{\epsfig{file =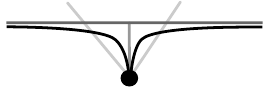,width=6.0cm,angle=0} }}
\vspace{-30pt}
\end{center}
\caption{The curve $\gamma$, the path $c$ and the arc $a$} \label{fig:arcs}
\end{figure}

\end{proof}
 
Using this lemma and the upper bound on flip distance in terms of intersection number, we can establish the following.

\begin{lemma}\label{lem:globalupper}
For $\Sigma$, $\Omega$ and $\Gamma$ as above:
$$
\diam\left( \MF(\Sigma) \right) \leq \diam\left( \MF(\Omega) \right)+ \diam\left( \MF(\Gamma) \right)+ 2(\kappa-n +1).
$$
\end{lemma}

The above inequality will allow us to treat the upper bounds by treating $\diam\left( \MF(\Gamma)\right)$ and $\diam\left( \MF(\Omega) \right)$ separately. We begin with the former.

\subsection{Upper bounds in terms of genus }

As above, $\Gamma$ is a genus $g\geq 1$ surface with a single boundary curve and a single marked point on the boundary. Our goal here is to show the following result.

\begin{theorem}\label{thm:uppergenus}
The diameter of the modular flip graph of $\Gamma$ satisfies
$$
\diam\left( \MF(\Gamma) \right) < A g \log (g+1)
$$
where $A$ can be taken to be $1000$.
\end{theorem}

Before proving the theorem we'll need two topological lemmas.

\begin{lemma}\label{lem:uppergenus1}
Let $T$ be a triangulation of $\Lambda$, a genus $g\geq 1$ surface with a single boundary curve and all $k$ marked points on the boundary. Then there exists $a \in T$ such that $\Lambda \setminus a$ is connected and of genus $g-1$. 
\end{lemma}

\begin{proof}
Observe that for an arc $a\in T$, $\Lambda \setminus a$ being connected and of genus $g-1$ is equivalent (cutting along a separating arc does not reduce genus). We now claim that $T$ always contains a non-separating arc. As $\Lambda \setminus T$ is a collection of triangles, it is of genus $0$. Now as $g(\Lambda) \geq 1$, one of the arcs of $T$ must be non-separating, otherwise $\Lambda \setminus T$ would still have positive genus.
\end{proof}

\begin{lemma}\label{lem:uppergenus2}
Let $T$ be a triangulation of $\Lambda$, a genus $g\geq 0$ surface with two boundary curves, both with marked points, and all marked points on the boundary. Then there exists $a \in T$ such that $\Lambda \setminus a$ has only one boundary component.
\end{lemma}
\begin{proof}
All marked points are on the boundary so it is impossible to triangulate $\Lambda$ without a triangle that has vertices on both boundary components. To see this we can argue by contradiction. If this is not the case, then we can split the triangles into two non-empty groups depending on whether they have all of vertices on one or the other boundary curve. But as the surface is connected, there must be a triangle of the first group which shares an arc with a triangle of the second. Thus, they must also share vertices, a contradiction.
\end{proof}

We now proceed to the proof of Theorem \ref{thm:uppergenus}.

\begin{proof}[Proof of Theorem \ref{thm:uppergenus}]
Let $T$ be any triangulation of $\Gamma$. Denote by $a_0$ the arc that forms the boundary of $\Gamma$. 

The first step will be to divide the surface along an arc that has equal genus (or close to equal) on both parts. By Lemma \ref{lem:uppergenus1}, there is an arc $a_1 \in T$ such that $\Gamma \setminus a_1$ is of genus $g-1$. The resulting surface $\Gamma^{1}:=\Gamma \setminus a_1$ now has two boundary components, one consisting of two arcs and the other of a single arc. Now by Lemma \ref{lem:uppergenus2}, there exists $a_2\in T$ such that $\Gamma_2:=\Gamma^{1} \setminus a_2$ has a single boundary curve consisting of $5$ arcs. In short, we found two arcs of $T$ such that cutting along these arcs produces a surface of genus $g-1$ with a single boundary component with $4$ more arcs than the original surface $\Gamma$. We can iterate the above process at total of $\floor{ \frac{g}{2}}$ times to obtain a collection of $ 2\floor{ \frac{g}{2}}$ arcs such that cutting along these arcs results in a genus $g - \floor{ \frac{g}{2}}$ surface $\overline {\Gamma}$ with a single boundary curve formed by $1 + 4 \floor{ \frac{g}{2}} $ arcs. One of these is $a_0$. 

Denote $p_0$ and $p_0'$ the two vertices of $a_0$ on $\overline {\Gamma}$. We denote $b$ the unique loop based in $p_0$ homotopic to the boundary of $\overline {\Gamma}$ and $b'$ the arc from $p_0'$ to $p_0$ which forms a triangle with $a_0$ and $b$ (see Figure \ref{fig:thmgenus1}).

\begin{figure}[h]
\leavevmode \SetLabels
\L(.39*.97) $p_0$\\
\L(.602*.96) $p_0'$\\
\L(.51*.98) $a_0$\\
\L(.53*.825) $b$\\
\L(.574*.71) $b'$\\
\endSetLabels
\begin{center}
\AffixLabels{\centerline{\epsfig{file =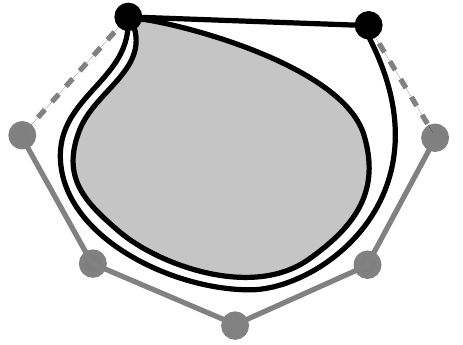,width=4.5cm,angle=0} }}
\vspace{-30pt}
\end{center}
\caption{$\overline{\Gamma}$} \label{fig:thmgenus1}
\end{figure}

Both $b$ and $b'$ have a nice property: they don't intersect $T$ too much. More precisely, as there are parallel to the boundary of $\overline{\Gamma}$ which is formed by arcs of $T$, they intersect each of the the remaining arcs at most twice. Thus 
$$
i(x, T) \leq 2 (\kappa(\Gamma) - 2 \left \lfloor \frac{g}{2} \right \rfloor ), \,\,x=b,b'.
$$
Now $\kappa(\Gamma) = 6g -2$ so we can deduce that 
$$
i(b,T) + i(b',T) \leq 20 g - 4.
$$
Now using the upper bound on the distance to a stratum in function of intersection number, we can introduce the arcs $b$ and $b'$ in at most $20 g - 4$ flips. 

The reason one might want to do this is that these arcs separate the surface into three canonical surfaces: a triangle containing $a_0$ and two surfaces with a single boundary curve and of genus $\floor{\frac{g}{2}}$ and $g - \floor{\frac{g}{2}}$. As such, up to homeomorphism, the pair of arcs $b$ and $b'$ are unique (see Figure \ref{fig:thmgenus2}). 

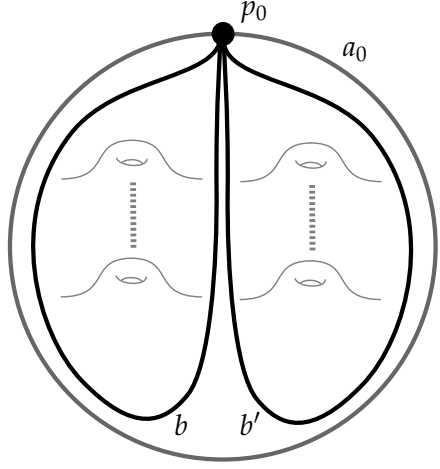
\begin{figure}[h]
\leavevmode \SetLabels
\L(.51*1.01) $p_0$\\
\L(.60*.92) $a_0$\\
\L(.453*.08) $b$\\
\L(.51*.08) $b'$\\
\endSetLabels
\begin{center}
\AffixLabels{\centerline{\epsfig{file =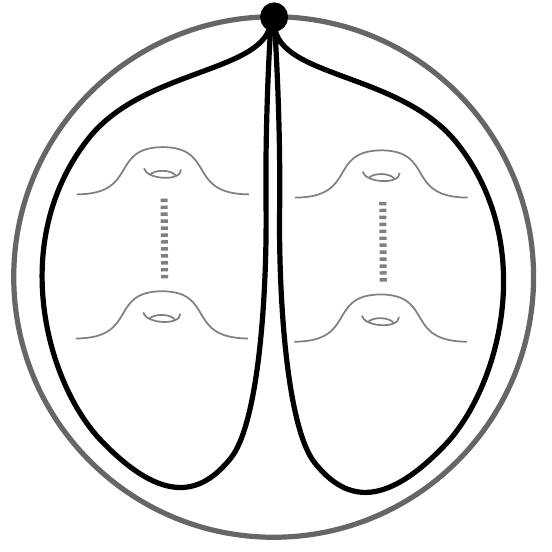,width=6.0cm,angle=0} }}
\vspace{-18pt}
\end{center}
\caption{The arcs $b$ and $b'$} \label{fig:thmgenus2}
\end{figure}

With this in hand, we will prove the bound by induction. We begin by checking the result for $g=1$. Here we need to check that the diameter is at most $1000 \log(2) > 693 > 5$. But there are at most $5$ different possible triangulations. Indeed such a triangulated surface is obtained by pasting four sides of a triangulated $5$-gon together. There are $C_3= 5$ different possible triangulations of the $5$-gon and only one to paste together the $5$-gon to get a one holed torus.

We now suppose $g(\Gamma)\geq 2$. 

Given two different triangulations $S$ and $T$ in $\MF(\Gamma)$, we flip both triangulations to obtain triangulations $S'$ and $T'$ with arcs as above. These triangulations now both belong to a stratum of $\MF(\Gamma_{b,b'})$ where $b$ and $b'$ are as above. We denote $\Gamma_1$ and $\Gamma_2$ the two non-triangular surfaces in $\Gamma \setminus \{b \cup b' \}$. Denote (for $k=1,2$) $S'_k$, resp. $T'_k$, the restrictions of $S'$, resp. $T'$, to $\Gamma_k$. We shall now flip $S'_k$ and $T'_k$ inside $\MF(\Gamma_k)$ for $k=1,2$. Once the triangulations coincide on both $\Gamma_1$ and $\Gamma_2$. they will coincide on $\Gamma$. 

By induction for $k=1,2$:
$$d(S'_k,T'_k) \leq \diam (\MF(\Gamma_k) )\leq A\, \frac{g+1}{2} \log \frac{g+3}{2}.$$

By induction (here we take into account that $g$ can be odd in the bound of $g- \floor{\frac{g}{2}}$):
$$d(S'_1,T'_1) \leq \diam (\MF(\Gamma_1)) \leq A\, \frac{g}{2} \log \frac{g+2}{2}$$
and
$$
d(S'_2,T'_2) \leq \diam (\MF(\Gamma_2)) \leq A\, \frac{g+1}{2} \log \frac{g+3}{2}.
$$

Putting this all together:
\begin{eqnarray*}
d(S,T) & \leq & d(S,S')+d(T,T') + d(S'_1,T'_1)+ d(S'_2,T'_2)\\
& \leq & 40 g - 8 + A\, \frac{g}{2} \log \frac{g+2}{2} + A\, \frac{g+1}{2} \log \frac{g+3}{2}\\
& \leq &A \, g \log (g+1).\\
\end{eqnarray*}
The last inequality can be checked via a computation using $A=1000$ and $g\geq 2$. 
\end{proof}

\begin{remark}\label{rem:onepuncture}

In light of Lemma \ref{lem:globalupper}, in the above theorem we've treated the case where the boundary of $\Gamma$ is a loop. The above proof however applies verbatim to the case where $\Gamma$ has a single puncture and no other boundary. The resulting theorem is the following.

\begin{theorem}\label{thm:onepunctureupper}
If $\Gamma$ is a surface with genus $g$ and one puncture, then the diameter of the modular flip graph of $\Gamma$ satisfies
$$ \diam (\MF(\Gamma)) < A g \log(g+1) $$ 
where $A$ can be taken to be $1000$.
\end{theorem}

\end{remark}

\subsection{Upper bounds in terms of number of punctures}

We now focus our attention on the flip graph of $\Omega$, a disk with $n-1$ interior punctures and one marked point on the unique boundary curve of $\Omega$. Our goal is to prove the following upper bound which is very similar to the upper bound for $\Gamma$.

\begin{theorem}\label{thm:upperpuncture}
If $\Omega$ has $n-1$ labelled punctures then the diameter of the modular flip graph of $\Omega$ satisfies
$$
\diam\left( \MF(\Omega) \right) < A n \log (n+1)
$$
where $A$ can be taken equal to $400$.
\end{theorem}

Before proceeding to the proof, we state a preliminary lemma. 

\begin{lemma}\label{lem:upperpuncture}
Let $T$ be a triangulation of $\Lambda$, a $m\geq 1$-punctured disk with $k\geq 1$ marked points on the boundary. Then $T$ contains an arc $a$ between an interior puncture and marked point on the boundary. 
\end{lemma}

\begin{proof}
If not, then a simple curve parallel to boundary does not intersect $T$ and hence $\Lambda \setminus T$ contains an embedded annulus.
\end{proof}

With that observation in hand, we now proceed to the proof of Theorem \ref{thm:upperpuncture}.

\begin{proof}[Proof of Theorem \ref{thm:upperpuncture}]
Let $T$ be a triangulation of $\Omega$ where we suppose that $n\geq 2$ (if $n=1$ then there the flip graph has a single triangulation). We denote the boundary arc of $\Omega$ $a_0$, the boundary marked point $p_0$ and the remaining punctures $p_j$, $j=1,\hdots,n$. Our goal will be to flip our triangulation to a canonical triangulation and argue by induction on the distance to this canonical triangulation. The upper bound on distance between arbitrary triangulations is then at most twice this distance. Our canonical triangulation $S$ is the following.

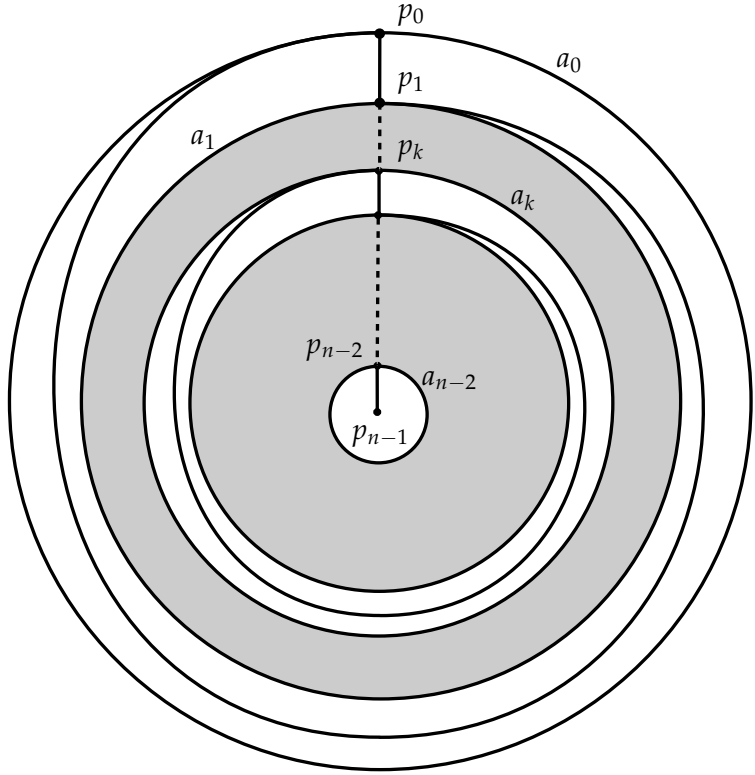
\begin{figure}[h]
\leavevmode \SetLabels
\L(.51*1.01) $p_0$\\
\L(.51*.92) $p_1$\\
\L(.65*.945) $a_0$\\
\L(.33*.843) $a_1$\\
\L(.61*.763) $a_k$\\
\L(.51*.83) $p_k$\\
\L(.43*.565) $p_{n-2}$\\
\L(.47*.45) $p_{n-1}$\\
\L(.532*.523) $a_{n-2}$\\
\endSetLabels
\begin{center}
\AffixLabels{\centerline{\epsfig{file =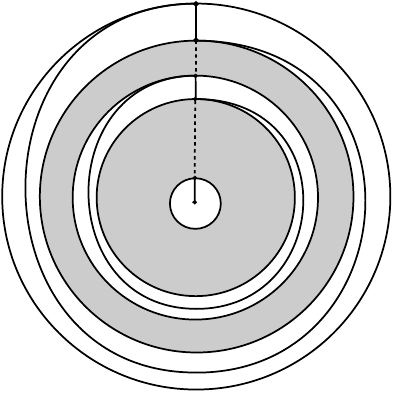,width=10.0cm,angle=0} }}
\vspace{-30pt}
\end{center}
\caption{The canonical triangulation} \label{fig:seashell}
\end{figure}

The triangulated surface is formed of layers. Each layer except for the last one is a cylinder with two boundary arcs $a_{k-1}$ and $a_{k}$ with punctures $p_{k-1}\in a_{k-1}$ and $p_{k}\in a_{k}$ for $k=1,\hdots,n-1$. The cylinders all contain a single interior arc from the triangulation as in the figure. The last layer is a disk with boundary $a_{n-1}$ and puncture $p_{n-1}\in a_{n-1}$ and interior puncture $p_n$. There is an arc in the triangulation between $p_{n-1}$ and $p_n$. 

To reach this triangulation from $T$ we proceed as follows. We begin by finding arcs that will divide the surface into punctured disks with the same (or close to the same) number of punctures in each disk. By Lemma \ref{lem:upperpuncture}, there is an arc $c\in T$ such that $\Omega\setminus c$ is a disk with $3$ boundary arcs: the arc $a_0$ and the two copies of $c$. We reiterate the above process $\floor{\frac{n}{2}}$ times cutting along $\floor{\frac{n}{2}}$ arcs to obtain a disk $\overline{\Omega}$ with $1 + 2 \floor{\frac{n}{2}}$ boundary arcs. On this boundary, $a_0$ joins two vertices: $p'_0$ and another, say $p_0''$, both copies of $p_0$. Consider the arc $b$ which forms a loop in $p'_0$ parallel to the boundary of $\overline{\Omega}$. Similarly, consider $b'$ which forms a triangle with $a_0$ and $b$: $b'$ is an arc between $p_0'$ and $p_0''$ which runs parallel to the boundary of $\overline{\Omega}$.

Both $b$ and $b'$ have a nice property: they don't intersect $T$ too much. More precisely, as there are parallel to the boundary of $\overline{\Omega}$ which is formed by arcs of $T$, they intersect each of the the remaining arcs at most twice. Thus 
$$
i(x, T) \leq 2 (\kappa(\Omega) - \left \lfloor{ \frac{n}{2}} \right \rfloor), \,\,x=b,b'.
$$
Now $\kappa(\Omega) = 3n-2$ and $- 2 \floor{ \frac{n}{2}} \leq -n+1$ so we can deduce that 
$$
i(b,T) + i(b',T) \leq 10n -10.
$$
Now using the upper bound on the distance to a stratum in function of intersection number, we can introduce the arcs $b$ and $b'$ in at most $10n - 10$ flips. 

The resulting triangulation now has an arc surrounding $\floor{ \frac{n}{2}}$ punctures, another surrounding $n - \floor{ \frac{n}{2}}$ punctures and the two arcs form a triangle with $a_0$ (see Figure \ref{fig:dividepunctures}).
\begin{figure}[h]
\leavevmode \SetLabels
\L(.51*1.01) $p_0$\\
\L(.60*.92) $a_0$\\
\L(.453*.08) $b$\\
\L(.51*.08) $b'$\\
\endSetLabels
\begin{center}
\AffixLabels{\centerline{\epsfig{file =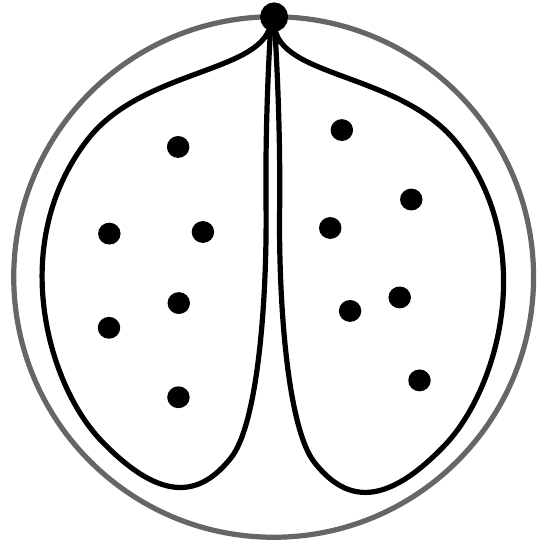,width=6.0cm,angle=0} }}
\vspace{-24pt}
\end{center}
\caption{The arcs $b$ and $b'$} \label{fig:dividepunctures}
\end{figure}
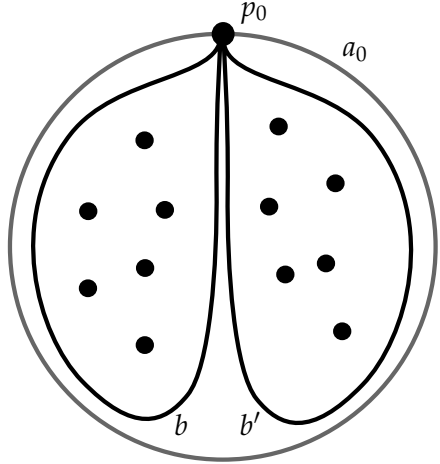
We now argue by induction on the two subsurfaces $\Omega_b$ and $\Omega_{b'}$ surrounded by $b$ and $b'$ to flip them towards their canonical triangulations. We have no control over which punctures are found in $\Omega_b$ and $\Omega_{b'}$ but the punctures do inherit an order from $\Omega$ and their canonical triangulations are meant with respect to that order. The number of flips inside each of the two subsurfaces, by induction, is at most
$$
A (\left \lfloor{ \frac{n}{2}} \right \rfloor + 1) \log(\left \lfloor{ \frac{n}{2}} \right \rfloor +2).
$$

Denote the resulting triangulation $T'$. We now need to merge the two subtriangulations of $T'$ to obtain the canonical one. To do this we proceed by steps where each step in the process will be to add a cylinder bounded by arcs $a_{k-1}$ and $a_k$ with punctures $p_{k-1}$ and $p_k$.

We begin with the first step. Puncture $p_1$ is either found in $\Omega_b$ (the lefthand subsurface) or in $\Omega_{b'}$ (the righthand subsurfaces). In either event it shares an arc with $p_0$ as both sub triangulations are canonical (and thus ordered). If $p_1$ on the left, we flip as in Figure \ref{FLIPLEFT}, and similarly if $p_1$ is on the right. As illustrated in the figures, the process takes 6 flips. We've constructed the first ring of the canonical triangulation. This ring surrounds a divided subsurface and we are in the same situation as above, where $p_1$ and $a_1$ play the part of $p_0$ and $a_0$ and with one less interior puncture. We can iterate the process a total of $n-1$ times (the last step is automatic) and arguing by induction we have reached the canonical triangulation in $6(n-1)$ steps from $T'$. 

\begin{figure}[h]
\centerline{ \epsfig{file =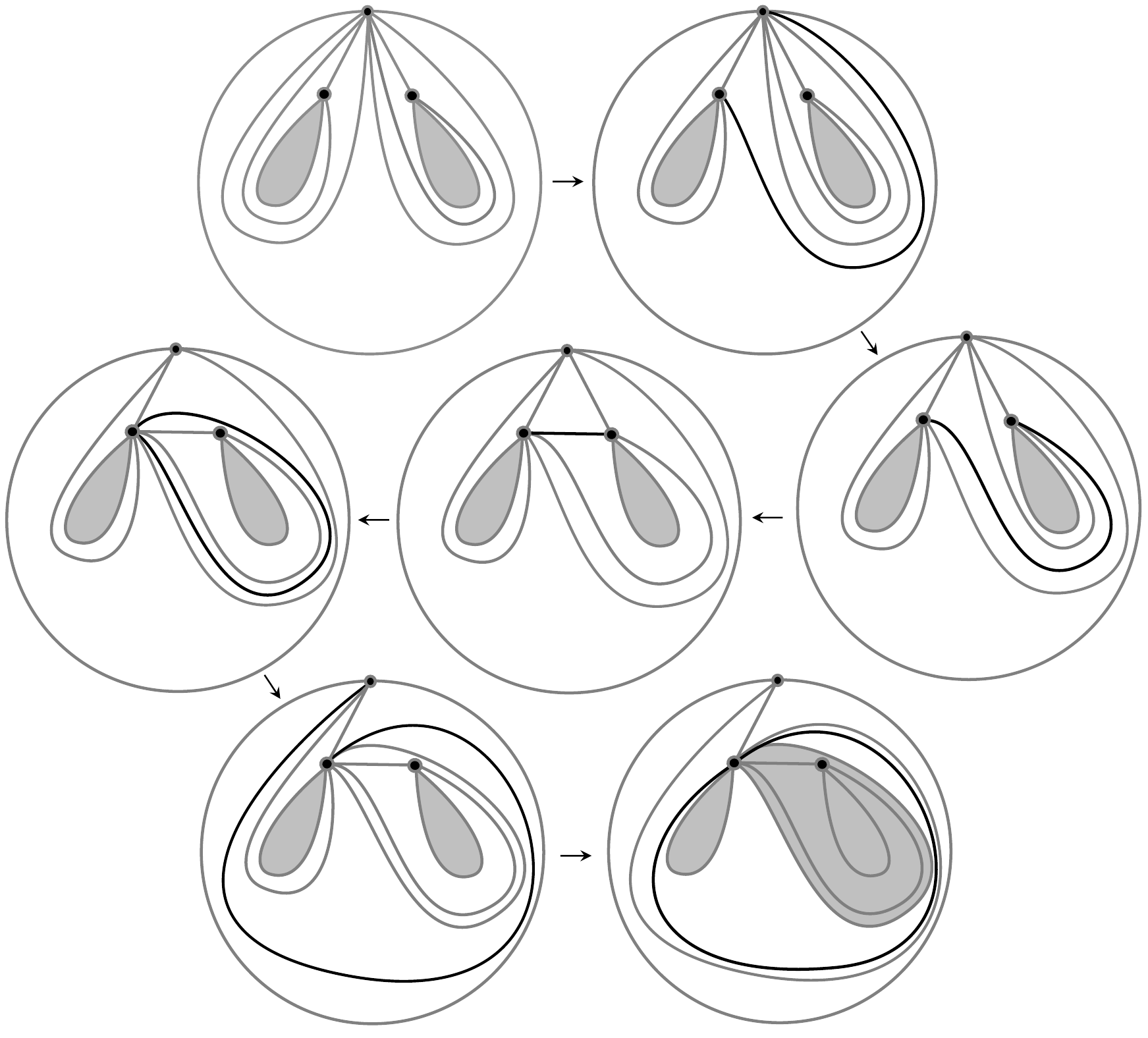, width=8.0cm, angle=0} }
\caption{Merging step}\label{FLIPLEFT}
\end{figure}

Putting this all together we have that for any $T\in \MF(\Omega)$
$$
d(T,S) \leq 10n - 10 + A (\left \lfloor{ \frac{n}{2}} \right \rfloor + 1) \log(\left \lfloor{ \frac{n}{2}} \right \rfloor +2) + 6(n-1).
$$
Arguing like in the genus case (see the proof of Theorem \ref{thm:uppergenus}) we obtain that 
$$
d(T,S)\leq A n \log(n+1).
$$
This shows that any two triangulations are at distance at most $2 A n \log(n+1)$ where $A$ can be taken equal to $200$.
\end{proof}

\begin{remark}\label{rem:unmarked}
The upper bound on $\MF(\Omega)$ is much easier if the punctures are unlabelled. Indeed, given a vertex $p$, if a triangulation contains arcs that are not incident to $p$, you can always find a flip that increases the incidence in $p$. 
Let $S$ and $T$ be two triangulations. After at most $4 \kappa - 2n$ valence-increasing flips, both $T$ and $S$ look like in Figure \ref{fig:flower_0}, that is, up to homeomorphisms they differ only in the shaded area. The shaded area can be thought as a triangulated $n$-agon. By Theorem \ref{th:STT} $T$ and $S$ differ by at most $4 \kappa - 2n + 2n = 4 \kappa$ flips. 

\begin{figure}[h]
\leavevmode \SetLabels
\L(.505*1.01) $p$\\
\L(.595*.93) $ $\\
\endSetLabels
\begin{center}
\AffixLabels{\centerline{\epsfig{file =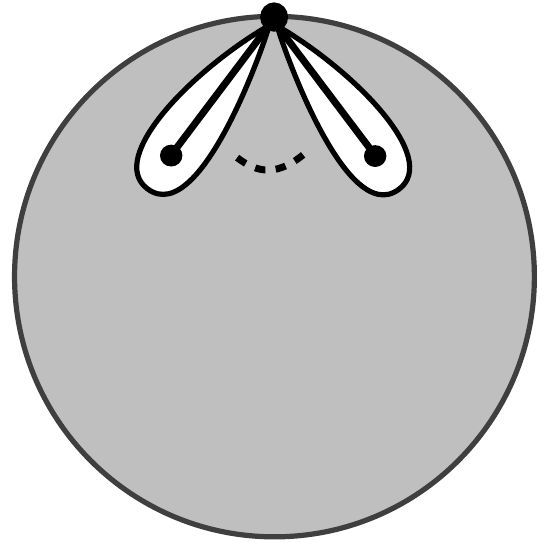,width=3.0cm,angle=0} }}
\vspace{-30pt}
\end{center}
\caption{The shaded area is triangulated.} \label{fig:flower_0}
\end{figure}
\begin{theorem}\label{thm:unmarked}
If $\Omega$ has $n-1$ unlabelled punctures then the diameter of the modular flip graph of $\Omega$ satisfies 
$$\diam (\MF(\Omega)) < A n $$
where $A$ can be taken equal to $12$. 
\end{theorem}
\end{remark}

\begin{remark}\label{rem:allpunctures}
The above proof 
however applies verbatim to the case where $\Omega$ is a punctured sphere (in this case the arcs $b$ and $b'$ in Figure \ref{fig:dividepunctures} coincide.)
We thus have the following.
\begin{theorem}\label{thm:genuszeroupper}
If $\Omega$ is a sphere with $n$ labelled punctures then 
$$\diam (\MF(\Omega)) < A n \log(n) $$
where $A$ can be taken equal to $410$. 
\end{theorem}
\begin{theorem}\label{thm:genuszeroupperunlabelled}
If $\Omega$ is a sphere with $n$ unlabelled punctures then 
$$\diam (\MF(\Omega)) < A n $$
where $A$ can be taken equal to $22$. 
\end{theorem}

\end{remark}

\subsection{Lower bounds via counting arguments}\label{ss:lower}

We now focus on lower bounds. They will essentially follow from a theorem of Sleator, Tarjan and Thurston \cite{STT1} and a counting argument. 

We begin with the following general lemma which follows from a theorem on grammars on graphs \cite{STT1}.

\begin{lemma}\label{lem:graphgrammar}
Let $\Lambda$ be a surface with $n$ punctures and $\MF(\Lambda)$ its modular flip graph.

Then for a fixed triangulation $T_\mu\in \MF(\Lambda)$ we have: 
$$\card \{ T \in \MF(\Lambda) \,| \, d(T,T_\mu) \leq m\} \leq 4^{10 m} 4^{\tilde{\kappa}(\Lambda)}.$$ 

\end{lemma}

\begin{proof}
This is a direct consequence of Theorem 2.3 of \cite{STT1} and the discussion in Section 5 in \cite{STT1}. For any triangulation $T$ one can construct its dual graph $G(T)$ (see Figure \ref{DualFlip}). 

\begin{figure}[htbp]
\centerline {\epsfig{file =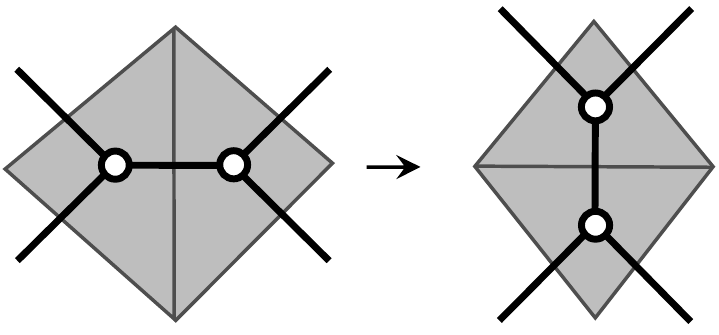,width=5.0cm,angle=0}}
\caption{The graph dual to a triangulation}\label{DualFlip}
\end{figure}

The graph $G(T)$ is a trivalent graph that has exactly $\tilde{\kappa}(\Lambda)$ vertices. We consider the three half-edges incident to a vertex labelled by the integers $1,2,3$ in clockwise order. Changing $T$ by one flip is equivalent to evolve $G(T)$ according the grammar in Figure \ref{flip_grammar}. 

\begin{figure}[h]
\leavevmode \SetLabels
\L(.41*.92) $1$\\
\L(.34*.62) $3$\\
\L(.455*.62) $2$\\
\L(.595*.92) $3$\\
\L(.525*.62) $2$\\
\L(.64*.62) $1$\\
\L(.39*.18) $3$\\
\L(.335*.3) $2$\\
\L(.335*.18) $1$\\
\L(.41*.29) $3$\\
\L(.465*.3) $1$\\
\L(.465*.18) $2$\\
\L(.57*.27) $3$\\
\L(.602*.45) $2$\\
\L(.563*.45) $1$\\
\L(.597*.2) $3$\\
\L(.601*.015) $1$\\
\L(.563*.015) $2$\\
\endSetLabels
\begin{center}
\AffixLabels{\centerline{\epsfig{file =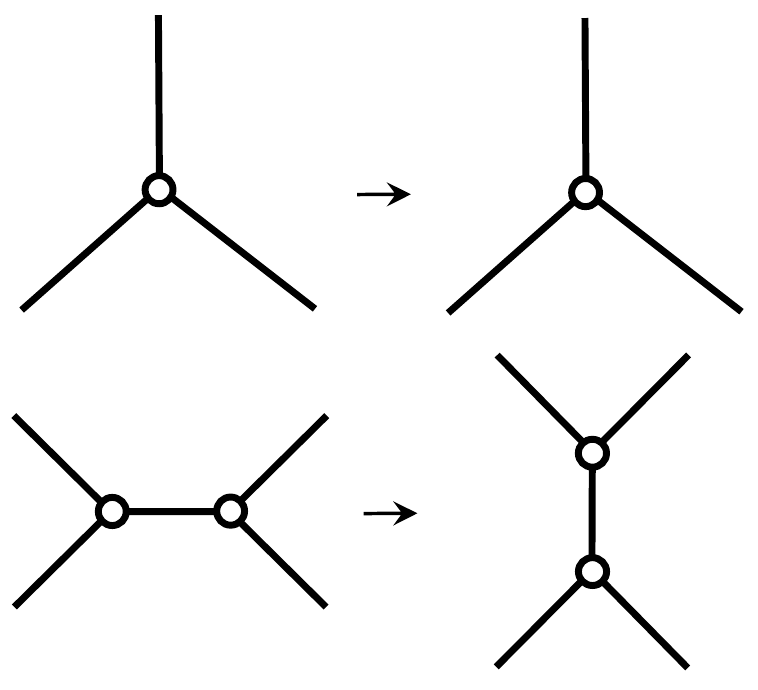,,width=5.0cm,angle=0} }}
\vspace{-30pt}
\end{center}
\caption{The grammar of a flip}\label{flip_grammar}
\end{figure}
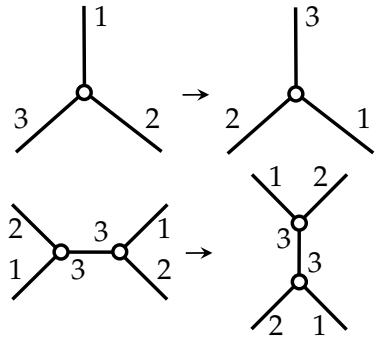

This grammar has two productions: one for doing the flip and the other for preparing the half-edges labels to allow the flip. Indeed, one flip on $T$ corresponds to perform at most 5 productions on $G(T)$: two to prepare the half-edge labels on the first vertex, two to prepare the half-edge labels on the second vertex, and one for the flip. It follows that the number of triangulations that can be obtained from $T_\mu$ in at most $m$ flips is bounded above by the number of graphs that can be derived by $G(T_\mu)$ with at most $10m$ productions. The latter is bounded above by $4^{10m}4^{\tilde{\kappa(\Lambda)}}$ by a straightforward application of Theorem 2.3 \cite{STT1} to the grammar we described. The same proof works verbatim for $T_\nu$.
\end{proof}

\begin{remark}
Setting $m = \diam\left(\MF(\Sigma)\right)$ in the lemma above, and then solving for $m$ using Inequality \ref{eq:card}, one obtains the following result:
\begin{corollary}\label{cor:card}
Let $\Sigma$ be a surface of genus $g$ with $n$ marked points, $\Gamma$ be a surface of genus $g$ with one boundary component and exactly one marked point on it, and $\Omega$ be a disk with $n-1$ interior punctures. We have: 
$$ \diam\left(\MF(\Sigma)\right) > \frac{\log (\card(\MF(\Gamma))) + \log (\card(\MF(\Omega))) - \tilde{\kappa}(\Sigma)\log(4)}{10 \log(4)} $$ 

\end{corollary}

\end{remark}

We now count vertices of our combinatorial moduli spaces.

\begin{lemma}\label{lem:counting}
Let $\Gamma$ be a surface of genus $g\geq 2$ with a single boundary loop and one marked point on the boundary. Let $\Omega$ be a disk with a single boundary component with a marked point on the boundary and with
$n-1$ interior labelled points. 
Then
\begin{eqnarray}
 \label{card1} \card \{ \MF(\Gamma)\} &\geq& \frac{g-1}{2} \, (g-1)!\\ 
 \label{card2} \card \{ \MF(\Omega)\} &\geq& C_{n-2}\, (n-1)!
\end{eqnarray}
where $C_{k}$ is the $k$-th Catalan number. \end{lemma}

\begin{proof}
We begin with Inequality \ref{card1}. For a given triangulation $T \in \MF(\Gamma)$, if we collapse the triangle which contains the boundary arc by cutting the triangle and pasting the two loose arcs together, we obtain a triangulated surface of genus $g$ with a single marked point. If you perform this on two triangulations $S,T \in \MF(\Gamma)$ and obtain different triangulations up to homeomorphism, then the triangulations we necessarily different to begin with. As such, there are at least as many triangulations in $\MF(\Gamma)$ then triangulations of a genus $g$ surface with a single marked point. It is a result of Penner \cite{PenWP} that there are at least
$$
\frac{g-1}{2} (g-1)!
$$
such triangulations and so the inequality is proved.

For Inequality \ref{card2} we argue as follows. Denote by $p_0$ the marked point on the boundary curve and $a_0$ the boundary loop. We begin by considering only triangulations where each interior puncture is surrounded by a single loop based at $p_0$ (see Figure \ref{fig:flower_0}). 


For two triangulations to be the same, they must coincide on the exterior of these loops. Cutting along the loops, one obtains an $n$-gon with one privileged side $a_0$. As such, we are in the classical case of counting triangulations of a polygon with an order on the sides and there are $C_{n-2}$ such triangulations. Any permutation of the vertex labelling gives a different polygon and thus we obtain the stronger lower bound
$$
(n-1)! \, C_{n-2}.
$$
\end{proof}

From this we obtain the following lower bound. 

\begin{corollary}\label{cor:count}
Let $\Sigma$ be a surface with $n$ labelled punctures. We have
$$\diam\left(\MF(\Sigma)\right)> B \left (n \log(n+1) + g \log(g+1) \right), $$ 
where $B$ can be taken equal to $2 \cdot 10^{-5}$.
\end{corollary}

\begin{proof}
We will use the following inequalities: 
\begin{enumerate}
\item $\log( C_n) > n $;
\item $\log n! > n \log(n) - n $. 
\end{enumerate}
Assume $n\geq 3$ and $g \geq 3$. From Lemma \ref{lem:counting} we get: 
\begin{align}
\log (\card(\MF(\Gamma))) &> \log (g-1)! > (g-1) \log(g-1) - g \\ 
\log(\card(\MF(\Omega))) &> \log(n-1)! > (n-1) \log(n-1) -n 
\end{align}
Assume that the punctures of $\Sigma$ are labelled. Plugging in the inequality in Corollary \ref{cor:card} we have:
\begin{align*} 
 \diam\left(\MF(\Sigma)\right) & > \frac{(g-1) \log(g-1) -g + (n-1) \log(n-1) - n - \tilde{\kappa}(\Sigma) \log(4)}{10 \log(4)} \\ 
& > \frac{(g-1) \log(g-1) - g + (n-1) \log(n-1) - n - (4g + 2n - 6) \log(4) }{10 \log(4)} \\
& > \frac{(g-1)\log(g-1) - (4 \log(4) +1)g}{10\log(4)} + \frac{(n-1) \log(n-1) - (2\log(4) +1)n}{10\log(4)} \\ 
& > B\, (g \log(g+1) + n\log(n+1)) 
\end{align*}
where $B$ can be taken to be $2 \cdot 10^{-5}$ and for $g \geq 705$ and $n \geq 50$. 
It is immediate to verify that 
$$
\diam\left(\MF(\Sigma)\right) > B\, (g \log(g+1) + n\log(n+1)) 
$$
also holds in the remaining cases ($g\leq 704$ or $n\leq 49$). 
\end{proof}

We note that we can improve the constant $B$ by conditioning $g$ and $n$ (giving them both lower bounds) but our principle interest is in the order of growth.

We obtain a similar result on lower bounds for unlabeled marked points. 
\begin{corollary}
If $\Sigma$ has $n \geq 511$ unlabelled punctures and is of genus $g$ then
$$\diam\left(\MF(\Sigma)\right) > B (g \log(g+1) + n )$$
where $B$ can be taken to be $10^{-3}$ . 
\end{corollary}

\begin{proof}
The graph grammar described in Lemma \ref{lem:graphgrammar} can be refined (see \cite{STT1} for details) so that
$$ \card{\MF(\Sigma)} \leq 3^{\tilde{\kappa}(\Sigma)} 8^m .$$ 
We have: 
$$ m \geq \frac{\log( \card(\MF(\Sigma)) - \tilde{\kappa}(\Sigma) \log(3)}{\log(8)}.$$
Let $\tilde{\Omega}$ be a disk with a single boundary component with a marked point on the boundary and with $n-1$ interior unlabelled points. As in Lemma \ref{lem:graphgrammar} we have
$$ \card{\MF(\Sigma)} \geq \card{\MF(\Gamma)} \, \card{\MF(\tilde{\Omega})} .$$ 
Now we use a result of Brown \cite{Brown} that provides lower bounds on the cardinality of $\MF(\tilde{\Omega})$:
$$ \card{\MF(\tilde{\Omega})} > \frac{2(4n - 7)!}{(n-1)! (3n - 4)!} .$$
An explicit computation shows that, for $n \geq 511$, the following holds: 
$$ \card{\MF(\tilde{\Omega})} > (9.1)^{n} > 3^{2n} .$$ 
From this we can conclude that
\begin{align*}
\diam\left(\MF(\Sigma)\right) & > \frac{(g-1) \log(g-1) -g - \tilde{\kappa}(\Sigma) \log(3)}{\log(8)} \\ 
& > \frac{(g-1) \log(g-1) - g + \log(9.1) n - (4g + 2n - 6) \log(3) }{\log(8)} \\
& > \frac{(g-1)\log(g-1) - (4 \log(3) +1)g}{\log(8)} + \frac{\log(9.1) n - \log(9)n }{\log(8)} \\ 
& > B (g \log(g+1) + n) 
\end{align*}
where the latter inequality holds for $g \geq 705$ and $B$ can be taken to be equal to $10^{-3}$. The final assertion can be checked directly for the cases $g < 705$. 
\end{proof}
As before, we note that by putting lower bounds on $g$ and $n$, the constant $B$ can be improved. 
\addcontentsline{toc}{section}{References}
\bibliographystyle{Hugo}
\bibliography{flip_references}


{\em Addresses:}\\
Department of Mathematics, University of Fribourg, Switzerland \\
Indiana University, Bloomington IN, USA \\
{\em Emails:} \href{mailto:hugo.parlier@unifr.ch}{hugo.parlier@unifr.ch}, \href{mailto:valentina}{vdisarlo@indiana.edu}\\

\end{document}